\documentclass[a4paper]{article}

\usepackage[utf8]{inputenc}
\usepackage[T1]{fontenc}
\usepackage{tikz}
\usepackage{amsmath}
\usepackage{amsthm}
\usepackage{amsfonts}
\usepackage{mathrsfs}
\usepackage{csquotes}
\usepackage{graphicx}
\usetikzlibrary{cd}
\usetikzlibrary{calc}
\usepackage{enumerate}
\usepackage{hyperref}
\usepackage{cleveref}

\title{Wagoner's Complexes Revisited}
\author{Jeremias Epperlein}

\newtheorem{thm}{Theorem}[section]
\newtheorem{lem}[thm]{Lemma}
\newtheorem{defn}[thm]{Definition}

\newtheorem{exam}[thm]{Example}

\newtheorem{cor}[thm]{Corollary}
\newtheorem{rem}[thm]{Remark}

\newcommand{\setsep}{\:|\:}

\newcommand{\N}{\mathbb{N}}
\newcommand{\R}{\mathbb{R}}
\newcommand{\Z}{\mathbb{Z}}

\newcommand{\PP}{\mathcal{P}}
\newcommand{\NN}{\mathcal{N}}
\newcommand{\CG}{\mathcal{G}}
\newcommand{\CU}{\mathcal{U}}

\newcommand{\CR}{\mathcal{R}}

\newcommand{\eps}{\varepsilon}
\newcommand*{\simplex}{
   \raisebox{0pt}{%
   \begin{tikzpicture}[scale=0.28]
     \draw[] (0,0) -- (1,0) -- (1,1) -- (0,0);
   \end{tikzpicture}}\hspace{0.06cm}}

\DeclareMathOperator{\SSE}{\text{SSE}}

\DeclareMathOperator{\Aut}{Aut}

\DeclareMathOperator{\id}{id}
\DeclareMathOperator{\supp}{supp}
\newcommand{\loc}{\text{loc}}
\DeclareMathOperator{\source}{s}
\DeclareMathOperator{\target}{t}
\DeclareMathOperator{\pro}{\delta}
\DeclareMathOperator{\im}{im}
\DeclareMathOperator{\sgn}{sgn}

\begin{document}

\maketitle

\begin{abstract}
  We generalize Wagoner's representation of the automorphism group
  of a two-sided  subshifts of finite type as
  the fundamental group of a certain CW-complex to 
  groupoids having a certain refinement structure.
  This significantly streamlines the original proof and allows
  us to extend this construction to, e.g., the automorphism group
  of subshifts of finite type over arbitrary finitely generated groups
  and the automorphism group of $G$-SFTs.
\end{abstract}

\section{Introduction}
\label{sec:intro}

The classification of two-sided subshifts of finite type
up to topological conjugacy is a fundamental open problem in
symbolic dynamics. The basic starting point
for this is Williams seminal paper
\cite{williamsClassificationSubshiftsFinite1973}
where he showed that every topological conjugacy between
edge shifts can be decomposed into a series of state-splittings and
-amalgamations. This allows for a algebraic formulation of topological
conjugacy in the form of strong shift equivalence of the matrices
defining the subshifts. Every conjugacy
can be represented as a series of
elementary strong shift equivalences.
This decomposition however is non-unique.

In a series of papers
\cite{wagonerMarkovPartitionsK21987},
\cite{wagonerTriangleIdentitiesSymmetries1990},
\cite{wagonerEventualFiniteOrder1990},
\cite{wagonerHigherDimensionalShiftEquivalence1990}
Wagoner introduced certain CW-complexes (more precisely geometric
realizations of simplicial
complexes) to deal with this ambiguity.
One of these complexes is the
complex of strong shift equivalences
between $\{0,1\}$-matrices.
Every edge path in this complex corresponds to
a sequence of elementary strong shift equivalences
and two such paths are homotopic iff they
define the same conjugacy. The automorphism
group of a subshift of finite type
can thus be recovered as the fundamental group of this complex.
This can be used to construct homomorphism from this
automorphism group into simpler groups.
There are two main ways to do this.
By allowing a large set of matrices or gluing in more triangles
one can embed the complex in a larger one, thus getting
a homomorphism from one fundamental group to the other.
This is used for example in
\cite{wagonerEventualFiniteOrder1990}
to construct the dimension group representation in that way.

One can also define the homomorphism on elementary strong shift
equivalences based on the matrices describing them. Checking
well-definedness then has to be done only along
the boundary of the triangles in the complex.
This was done for the sign-gyration-compatibility-condition homomorphism
in \cite{kimAutomorphismsDimensionGroup1992}

The development culminated in the celebrated paper
\cite{kimWilliamsConjectureFalse1999}
by Kim and Roush
showing that shift equivalence does not imply strong shift
equivalence.
While there are multiple expositions (for example
\cite{boylePositiveKTheorySymbolic2002},
\cite{zbMATH05347058} and
\cite{wagonerStrongShiftEquivalence1999}) of this result
they all take the construction of the complexes as a black box
(\cite{wagonerStrongShiftEquivalence1999} gives the most details).

Wagoners investigation actually started with another, similarly defined complex,
which he called the space of Markov partitions.
We use the fact that (ordered) Markov partitions in this context
correspond to conjugacies to topological Markov shifts
to give a new and simplified definition of this space.

This has two advantages.
First of all it significantly shortens many of the proofs.
Additionally this construction makes sense
for any groupoid with a generating set of morphisms.
The complex corresponding
to the complex of strong shift equivalences
can then be interpreted as the nerve
of the groupoid with respect to the generating set.
This abstraction neatly isolates the assumptions
we need on the groupoid. We can also
apply this abstract result
to other situations in symbolic dynamics
and obtain an analogous complex
capturing conjugacies between
subshift of finite type over
finitely generated groups
and $G$-SFTs.

The paper is structured as follows.
We start in \Cref{sec:strong-shift-eq} by giving an exposition of the construction
in the classical case of two-sided
subshifts of finite type. There are no proofs
in this section but references to the corresponding
abstract versions of the results in the later sections.
Next we give a detailed description of the construction for arbitrary
groupoids in \Cref{sec:hug} and introduce a certain
refinement structure which allows us
to analyze the topological structure of our complex.
In \Cref{sec:markov-shifts-fg} we show
that we can construct such a structure
for Markov shifts over finitely generated groups.
The complex of elementary strong shift equivalences
is finally constructed
in \Cref{sec:matrix-eq} and there we show that it is isomorphic
to the complex constructed earlier.
In \Cref{sec:gsfts} we construct an analogous complex
for $G$-SFTs.
The next section \Cref{sec:degenerate} answers a question
of Boyle and Wagoner about the SSE complex with
possibly degenerate matrices.
Finally we show in \Cref{sec:contractability} that under certain additional restrictions
the analogue to the space of Markov partitions
we constructed is non only simply connected but even contractible.
For this we use a certain decomposition of the simplex
going back to Freudenthal which we cover in more detail in
\Cref{sec:freudenthal}.

\section{Strong shift equivalence}
\label{sec:strong-shift-eq}
Before we paint the picture in the abstract setting, we sketch the
way ahead in the familiar setting of two-sided subshifts of
finite time.
Let $A$ be a square $\{0,1\}$-matrix of size $n \times n$. We say that $A$ in
\emph{non-degenerate} if $A$ has no zero columns or rows.
We say that two such matrices $A$ and $B$
are \emph{elementarily strong shift equivalent}
if there are non-degenerate $\{0,1\}$-matrices $R,S$
with $A=RS$ and $B=SR$. The matrices $A$ and $B$ are
\emph{strong shift equivalent} if there
is a chain
of non-degenerate square matrices $A=C_1,C_2,\dots,C_n=B$
such that $C_{\ell}$ is elementarily strong shift equivalent to
$C_{\ell+1}$ for all $\ell \in \{1,\dots,n-1\}$.

Every non-degenerate square matrix defines a directed graph $G_A$ with
vertex set $\{1,\dots,n\}$ and edge set
$\{(i,j) \in \{1,\dots,n\}^2 \setsep A_{ij}=1\}$.  The set of
bi-infinite paths in this graph is called the \emph{vertex shift}
associated to $A$ and is denoted by $X_A$. Such shift spaces
are also called \emph{topological Markov shifts}. More precisely, we have
\[X_A = \{x \in \{1,\dots,n\}^\Z \setsep A_{x_\ell,x_{\ell+1}}=1
\text{ for all } \ell \in \Z\}.\]  Endowing $\{1,\dots,n\}$ with the
discrete topology and $\{1,\dots,n\}^\Z$ with the product topology
turns this into a metrizable, compact, totally disconnected space.
The left shift $\sigma$ defined by $\sigma(x)_\ell = x_{\ell+1}$ acts
continuously on it.

Two such vertex shifts $X_A$ and $X_B$ are \emph{topologically conjugate} to 
each other if there is a homeomorphism between them which commutes
with the shift.
To every elementary strong shift equivalence given by
matrices $R \in \{0,1\}^{m \times n} $ and $S \in \{0,1\}^{n \times m}$
between $A$ and $B$ one can associate
a map $\varphi_{R,S}: X_A \to X_B$ by
the condition that $\varphi_{R,S}(x)_i$ is
the unique element $a$ of $\{1,\dots,n\}$
such that $R_{x_i,a} S_{a,x_{i+1}}=A_{x_i,x_{i+1}}=1$.
This map clearly commutes with the shift and also turns out to be a
homeomorphism (see \Cref{thm:elementary-from-algebraic}).
We call such a conjugacy an \emph{elementary conjugacy}.

We will show that a conjugacy $\varphi$ is elementary if and only if
the value of $\varphi(x)_i$ only depends on $x_i$ and $x_{i+1}$
and the value of $\varphi^{-1}(y)_i$ only depends on $y_i$ and $y_{i-1}$
(see \Cref{thm:algebraic-from-elementary} and \ref{thm:elementary-from-algebraic}).
This characterization should be well-known but the author couldn't
find it precisely in the literature. However, it can be derived from
statements about Markov partitions in
\cite{wagonerStrongShiftEquivalence1999}
or the characterization of elementary conjugacies as
the composition of a out-splitting and
an in-amalgamation, see \cite[Proposition 7.2.11]{lindIntroductionSymbolicDynamics1995}

Williams showed that every conjugacy can be represented as a series
of elementary conjugacies and their inverses.
Hence $X_A$ and $X_B$ are conjugate whenever $A$ and $B$ are strongly
shift equivalent.

This representation however is non-unique. In order to cope with this,
we define a CW-complex of elementary conjugacies as follows. We start with isolated
points corresponding to vertex shifts
defined by non-degenerate square matrices as above. For every
elementary conjugacy between vertex shifts $X_A$ and $X_B$
we add a line segment. Notice that
there might be multiple such line segments between $X_A$ and $X_B$.
Furthermore for some line segments the start and end point might
coincide, i.e. there will be loops. For example the identity map
determines a loop at very vertex. Finally we glue a disk into
every triangle of the form 
\begin{center}
  \begin{tikzcd}
    & X_B \arrow[rd, "\varphi_2"] &\\
    X_A \arrow[ru, "\varphi_1"]
    \arrow[rr ,"\varphi_3" below]& & X_C
  \end{tikzcd}.
\end{center}
where $\varphi_1,\varphi_2$ and $\varphi_3$ are elementary conjugacies,
whenever the triangle commutes. Later we will also clue in higher
dimensional simplices, but in this section we are only interested in
the fundamental group of this space so this does not matter.

This construction resembles the construction of a nerve of a
category (see for example \cite{segalClassifyingSpacesSpectral1968},
a reference that is general enough to cover our case is
\cite[Definition 15.5]{kozlovCombinatorialAlgebraicTopology2008}), but notice that the elementary conjugacies are not closed
under composition. Now every homotopy class of paths between $X_A$ and
$X_B$ has a representative that is a concatenation of $m$ line
segments which we glued in and each of these line segments
corresponds to an elementary conjugacy $\varphi_i$ and a sign
$\varepsilon_i$ depending on the direction we pass through that
segment.  Then the concatenation $\varphi_n^{\varepsilon_n} \circ
\dots \circ \varphi_1^{\varepsilon_1}$ is a well defined conjugacy.
The result of this concatenation only depends on the homotopy
class of the path.
Thus we get a homomorphism from the fundamental group of our complex
based at $X_A$ to the automorphism group of $X_A$. 

Williams result from above
tells us that this homomorphism is surjective and Wagoner showed that
it is injective.

To do so,  Wagoner considered the space of Markov partitions.
A Markov partition of a space $X$ with a homeomorphism
$f: X \to X$ here means a partition of $X$ into finitely many clopen sets
$\CU=\{U_i \setsep i \in \{1,\dots,n\}\}$ such that for every
two-sided sequence
$(x_n)_{n \in \Z} \in \{1,\dots,n\}^\Z$ with
\begin{align*}
  U_{x_i} \cap f^{-1}(U_{x_{i+1}}) \neq \emptyset
\end{align*}
there is a unique element
in the intersection $\bigcap_{n \in \Z} f^{-n}(U_{x_n})$.
Every partition $\CU$ of $X$ into clopen sets gives rise
to
\begin{enumerate}[(a)]
\item 
  a matrix $A_\CU \in \{0,1\}^{n \times n}$ defined by
  $A_{i,j}=1$ iff $U_i \cap f^{-1}(U_j) \neq \emptyset$, and
\item
  to a continuous map
  from $X$ to the vertex shift $X_{A_{\CU}}$
  which maps the point $x$ to the sequence of partition elements
  that $x$ traverses, i.e. $x \mapsto (y_k)_{k \in \Z}$
  such that $f^k(x) \in U_{y_k}$ for all $k \in \Z$.
\end{enumerate}
A partition $\CU$ is a Markov partition 
iff this induced map is a conjugacy.
On the other hand every conjugacy
$\varphi: X_A \to X_B$ with $B \in \{0,1\}^{m \times m}$
induces a Markov partition $\CU$ of $X_A$
via $\CU=\{ \varphi^{-1}([k]_0) \setsep k \in \{1,\dots,n\}\}$
where $[k]_0 := \{x \in X_B \setsep x_0 = k\}$.
Thus there is a bijection between Markov partitions
of a vertex shift $X_A$ and conjugacies
from $X_A$ to other vertex shifts.

Hence we define the space of ordered Markov partitions of $X_A$ as follows.
We again build up our space from vertices, edges and triangles.
The vertices are conjugacies $X_A \to X_B$ for some vertex shift
$X_B$. We add an edge from $\varphi_1: X_A \to X_B$ to $\varphi_2: X_A
\to X_C$ whenever
$\varphi_2 \circ \varphi_1^{-1}: X_B \to X_C$ is an elementary
conjugacy. In this space
we can have at most two edges between vertices and
we have precisely one loop at every vertex as the idendity
is an elementary conjugacy.
This space of Markov partitions is a covering space
of the connected component of $X_A$ in the space we constructed
earlier (\Cref{thm:simply-connected}).

We will show that this space of Markov partitions is simply connected
(\Cref{thm:simply-connected} and \Cref{lem:refinment-ongroups}, in the later version we
also glue in higher simplices and this will make the space even
contractible, see \Cref{sec:contractability}). This
implies that two paths in the complex of
elementary conjugacies are homotopic iff
the corresponding conjugacies are the same, see \Cref{thm:covering}.

Finally we define yet another complex $\SSE(A,\{0,1\})$ whose vertices are
non-degenerate square $\{0,1\}$-matrices, whose edges correspond to
elementary strong shift equivalences and triangles correspond to
diagrams of the form
\begin{center}
  \begin{tikzcd}
    & B \arrow[rd, "{R_2, S_2}"] &\\
    A \arrow[ru, "{R_1, S_1}"]
    \arrow[rr ,"{R_3, S_3}" below]& & C
  \end{tikzcd}
\end{center}
where the following triangle equations
\begin{align*}
  R_1 R_2 &= R_3, \\
  R_2 S_3 &= S_1, \\
  S_3 R_1 &= S_2
\end{align*}
hold. We then show that the map sending a matrix to its corresponding
vertex shift and an elementary strong shift equivalence to the
corresponding elementary conjugacy induces a isomorphism between
these spaces (see \Cref{thm:simplicial-set-iso}).

This finally shows that the automorphism group of a vertex shift
can be represented as the fundamental group of such a very
algebraically defined topological space (see \Cref{cor:sse-automorphismgroup}).

\section{Homotopically Unique Generation}
\label{sec:hug}
Let $\Gamma$ be a groupoid. Denote its  objects by $\Gamma^0$ and
its morphisms by $\Gamma^1$. We say that $H \subseteq \Gamma^1$ \emph{generates}
$\Gamma$ if $H$ contains all identity morphisms and for every $\psi \in \Gamma^1$ there are $\varphi_1,\dots,\varphi_n \in H
\cup H^{-1}$
with $\psi=\varphi_n \circ \dots \circ \varphi_1$.
Notice that for morphisms in a groupoid we write composition
from right to left as is customary for function composition.
Denote by $s(\varphi)$ and $t(\varphi)$ the
source and target of the morphism $\varphi$.
We now define a $\Delta$-complex (see for
example \cite[Chapter 2.1]{hatcherAlgebraicTopology2002})
$\NN(\Gamma,H)$,
the \emph{nerve} of $(\Gamma ,H)$.
The $n$-simplices of $\NN(\Gamma,H)$ are 
ordered tuples of objects $(X_0,\dots,X_n)$
together with morphisms $(\varphi_{i,j})_{i \neq j}$ 
with $\varphi_{i,j}: X_i \to X_j$ such that
$\varphi_{j,k} \circ \varphi_{i,j} = \varphi_{i,k}$
whenever $i < j < k$. 

We will also write $\NN(\Gamma,H)$
for the geometric realization of $\NN(\Gamma,H)$
and this CW-complex is the object we will mainly talk
about. In the following it will be enough to think of $\NN(\Gamma,H)$
as a CW-complex.

Hence $\NN(\Gamma,\Gamma^1)$ is precisely the usual notion of
a nerve.

\begin{rem}
	An $n$-simplex in $\NN(\Gamma,H)$ is given by a sequence $\varphi_{1,2}, \dots,
	\varphi_{n-1,n}$ of composable morphisms in $H$ such that
	$\varphi_{\ell,\ell+1}\circ\dots\circ\varphi_{k,k+1}$ is
	again in $H$ for all $k,\ell$ with $k<\ell$.
\end{rem}

\begin{defn}
	We say that $H$ \emph{homotopically uniquely generates} $\Gamma$
	if $H$ generates $\Gamma$ and every
	two sequences $\varphi_1,\dots,\varphi_n$ and 
	$\varphi'_1,\dots,\varphi_{n'}'$ in $H \cup H^{-1}$ with
	$\varphi_1 \circ \dots \circ \varphi_n=\varphi'_1 \circ \dots \circ \varphi'_{n'}$
	are homotopic as paths in $\NN(\Gamma,H)$.
\end{defn}

Based on $(\Gamma,H)$ we can now define for every object $X \in \Gamma^0$
a $\Delta$-complex $\PP(\Gamma,H,X)$. We will later see
that this is the universal covering space for $\NN(\Gamma,H)$.

The vertices of $\PP(\Gamma,H,X)$ are morphisms $\varphi: X \to Y$.
The $n$-simplices of $\PP(\Gamma,H,X)$ are tuples $(\varphi_0,\dots,\varphi_n)$
such that for every $i,j$ we have $\varphi_j \circ  \varphi_i^{-1} \in
H$. To simplify notation we write $\varphi_1 \to \varphi_2$ if
$\varphi_2 \circ \varphi_1^{-1} \in H$ and $\varphi_1 - \varphi_2$
if either $\varphi_1 \to \varphi_2$ or $\varphi_2 \to \varphi_1$.

\begin{lem}\label{lem:lifting-triangles} 
	Let $\varphi_1,\varphi_2,\varphi_3$ be morphisms in $\Gamma$ with
   source $X$ such that 
	\begin{center}
		\begin{tikzcd}
			& \varphi_2 \arrow[rd] &\\
			\varphi_1 \arrow[ru] \arrow[rr]& & \varphi_3
		\end{tikzcd}.
	\end{center}
	is a triangle in $\PP(\Gamma,H,X)$. Then 
	\begin{center}
		\begin{tikzcd}
			& \target(\varphi_2) 
			\arrow[rd, "\varphi_3 \circ \varphi_2^{-1}"] &\\
			\target(\varphi_1) 
			\arrow[ru, "\varphi_2 \circ \varphi_1^{-1}"] 
			\arrow[rr ,"\varphi_3 \circ \varphi_1^{-1}" below]& & \target(\varphi_3)
		\end{tikzcd}.
	\end{center}
	is a triangle in $\NN(\Gamma,H)$.
\end{lem}
\begin{proof}
	The edges in this triangle are $1$-simplices in $\NN(\Gamma,H)$ by the 
	definition of $\PP(\Gamma,H,X)$. The
	triangle is a $2$-simplex since $(\varphi_3 \circ \varphi_2^{-1}) \circ (\varphi_2 \circ \varphi_1^{-1}) = \varphi_3 \circ \varphi_1^{-1}$. 
\end{proof}

Denote by $\pi_1(\NN(\Gamma,H))$ the
``combinatorial'' fundamental groupoid of $\NN(\Gamma,H)$
whose objects are vertices in this space (i.e. objects of $\Gamma$)
and whose morphisms are homotopy classes of paths. A more precise
but also more cumbersome notation would be $\pi_1(\NN(\Gamma,H),\Gamma^0)$.

\begin{thm}
  \label{thm:covering}
  The space $\PP(\Gamma,H,X)$ is a covering space of $\NN(\Gamma,H)$ for
  every $X$ via the covering map defined by
  $[\varphi_1,\dots,\varphi_n] \mapsto
  [t(\varphi_1),\dots,t(\varphi_n)]$.

  If $\PP(\Gamma,H,X)$ is simply
  connected then the map
  \[\omega: \pi_1(\NN(\Gamma,H)) \to \Gamma, (\varphi_1, \dots,
    \varphi_n) \mapsto \varphi_n \circ \dots \circ \varphi_1\] is an
  isomorphism.  In particular, in that case $H$ homotopically uniquely generates
  $\Gamma$.
\end{thm}
\begin{proof}
	The map $\omega$ is well defined since it is well defined 
	for the boundary of the triangles we glued in.
	The map is surjective since $H \cup H^{-1}$ is a generating set.
   Let $X,Y$ be objects in $\Gamma$.
   Since $H \cup H^{-1}$ is a generating set, $\Gamma$ is connected,
   so there is a morphism $\psi: X \to Y$ in $\Gamma^1$.
   This morphism induces an isomorphism of
   $\Delta$-complexes $\PP(\Gamma,H,X) \to \PP(\Gamma,H,Y)$
   via \[[\varphi_1,\dots,\varphi_n] \mapsto [\varphi_1
   \circ \psi^{-1}, \dots,\varphi_n \circ
   \psi^{-1}].\] Thus $\PP(\Gamma,H,X)$ is
   simply connected for all $X$ under the assumptions
   of the theorem.
   To see that $\omega$ is injective, let $(\varphi_1,\dots,\varphi_n)$ be a
   path in $\NN(\Gamma,H)$ such
   $\varphi_n \circ \dots \circ \varphi_1=\id$ in $\Gamma$.  This path
   lifts to the path
   $\varphi_1 - \varphi_2 \circ \varphi_1 - \dots - (\varphi_n
   \circ \dots\circ \varphi_1)$ in $\PP(\Gamma,H,X)$ which starts and
   ends in $\source(\varphi_1)=\target(\varphi_n)$.
   Since $\PP(\Gamma,H,X)$ is simply connected, this loop can be
   triangulated.  Now every triangle in $\PP(\Gamma,H,X)$ maps by
   \Cref{lem:lifting-triangles} to a triangle in $H$, hence the loop
   $(\varphi_1, \dots, \varphi_n)$ is contractible in $\NN(\Gamma,H)$.
\end{proof}

We are now looking for conditions under which $\PP(\Gamma,H,X)$ is simply
connected.

\begin{defn}
  Let $\Gamma$ be a groupoid and $H$ a generating set of $\Gamma$
  containing all identities.  Let $\cong$ be an equivalence relation
  between the morphisms of $\Gamma$, let $H_n \subseteq H^n$ be sets
  and let $\pro_n: H_n \to H$ be maps (we often drop the index $n$
  and simply write $\delta$ when the number of arguments is clear). We call the triple
  \[(\cong, (H_n)_{n \in \N}, (\pro_n)_{n \in \N})\] a \emph{refinement
    structure for $(\Gamma, H)$} if the following hold (all statements
  only have to hold if terms involving $\delta_n$ are actually
  defined, that is, if its arguments are in $H_n$).
  \begin{description}
 \item[Equivalent morphisms are exchangeable.]
    \begin{align}
      \varphi_1 \cong \varphi_2 \text{ and }
      \varphi_3 \cong \varphi_4  \text{ and }
      \varphi_1 \to \varphi_3
      &\implies \varphi_2 \to \varphi_4 
        \label{eq:conditions-equivalent}
    \end{align}
  \item[Refining one set doesn't change it.]
    \begin{align}
      H_1&=H \label{eq:conditions-triv-H} \\
      \delta(\varphi)&\cong\varphi \label{eq:conditions-triv-delta}\
    \end{align}
  \item[Permutations of the arguments don't matter.] Let $\kappa$ be a
    permutation of $\{1,\dots,n\}$.
    \begin{align}
      (\varphi_1,\dots,\varphi_n) \in H_n
      &\implies (\varphi_{\kappa(1)},\dots,\varphi_{\kappa(n)}) \in H_n
        \label{eq:conditions-permutation-H}  \\
      \delta_n(\varphi_1,\dots,\varphi_n)
      &\cong
        \delta_n(\varphi_{\kappa(1)},\dots,\varphi_{\kappa(n)}) \in H_n
        \label{eq:conditions-permutation-delta}
    \end{align}
  \item[Refinement allows grouping.]
    \begin{align}
      (\delta_n(\varphi^1_1,\dots,\varphi^1_n),\dots,
      \delta_n(\varphi^k_1,\dots,\varphi^k_{n}))\in H_k
      &
        \iff
        (\varphi_1^1,\dots,\varphi_n^1,
        \dots,
        \varphi_1^k,\dots,\varphi_n^k)
        \in H_{kn}
        \label{eq:conditions-grouping-H}\\
      \delta_k(\delta_n(\varphi^1_1,\dots,\varphi^1_n),
      \dots,
      \delta_n(\varphi^k_1,\dots,\varphi^k_{n}))
      &\cong
        \delta_{kn}(
        \varphi_1^1,\dots,\varphi_n^1,
        \dots,\varphi_1^k,\dots,\varphi_n^k)
        \label{eq:conditions-grouping-delta}
    \end{align}
  \item[Redundant arguments can be dropped.]
    \begin{align}
      (\varphi_{1},\varphi_1,\dots,\varphi_{n}) \in H_{n+1}
      &\iff (\varphi_1,\dots,\varphi_n) \in H_n
        \label{eq:conditions-drop-H}\\
      \delta_{n+1}(
      \varphi_1,\varphi_1,
      \varphi_2,\dots,\varphi_n)
      &\cong
        \delta_n(\varphi_1,\varphi_2,\dots,\varphi_n)
        \label{eq:conditions-drop-delta} 
    \end{align}
  \item[Arrows allow refinement.]
    \begin{align}
        %
      \varphi_1 \leftarrow \varphi_2 \rightarrow \varphi_3
      &\implies (\varphi_1,\varphi_2,\varphi_3) \in H_3
        \label{eq:conditions-arrow-2-H} \\
      \varphi_1 \rightarrow \varphi_2 \leftarrow \varphi_3
      &\implies (\varphi_1,\varphi_2,\varphi_3) \in H_3
        \label{eq:conditions-arrow-3-H} \\
      \varphi_1 \rightarrow \varphi_2 \text{ and }
      \varphi_3 \rightarrow \varphi_4
      &\implies
        \delta_2(\varphi_1,\varphi_3) \to \delta_2(\varphi_2,\varphi_4)
        \label{eq:conditions-arrow-delta}
    \end{align}
  \end{description}
\end{defn}

A consequence of these properties, which will we need later, is the
following.
\begin{lem}
  \label{lem:arrow-right-grouping}
  If there exits a refinement structure for $(\Gamma,H)$
  and $\psi_1,\dots,\psi_n \to \varphi$,
  then $(\psi_1,\dots,\psi_n, \varphi) \in H_{n+1}$
  and $\delta_{n+1}(\psi_1,\dots,\psi_n,\varphi) \to \varphi$.
\end{lem}
\begin{proof}
  We show this by induction on $n$, the case $n=0$ being nothing more then
  \eqref{eq:conditions-triv-delta}.
  Assume we showed the assertion for $n$.
  By \eqref{eq:conditions-arrow-2-H} and \eqref{eq:conditions-triv-delta} we know that
  \[(\delta(\psi_{n+1}),\delta(\psi_1,\dots,\psi_n, \varphi)) \in
  H_2\]
  hence by \eqref{eq:conditions-drop-H} and
  \eqref{eq:conditions-permutation-H}
  also $(\psi_1,\dots,\psi_{n+1},\varphi) \in H_{n+1}$.
  Similarly \eqref{eq:conditions-arrow-delta},
  \eqref{eq:conditions-permutation-delta},
  \eqref{eq:conditions-drop-delta} and \eqref{eq:conditions-triv-delta}
  imply that $\delta(\psi_1,\dots,\psi_{n+1},\varphi)=\delta(\psi_{n+1},\delta(\psi_1,\dots,\psi_n,\varphi)) \to \delta(\varphi,\varphi)=\varphi$.
\end{proof}

We are now ready to proof the main abstract result in this paper.
The proof closely follows that of Wagoner and Badoian in \cite{badoianSimpleConnectivityMarkov2000}.
\begin{thm}\label{thm:simply-connected}
  If there exists a refinement structure for $(\Gamma,H)$
  then $\PP(\Gamma,H,X)$ is simply connected for all objects $X$
  and therefore $H$ homotopically uniquely generates $\Gamma$.
\end{thm}
\begin{proof}
  Consider a loop $\varphi_1 - \cdots - \varphi_n - \varphi_1$ in $\PP(\Gamma,H,X)$.  Since
  every triangle with three equal vertices is in $\PP(\Gamma,H,X)$, we can add
  edges of the form $\varphi \to \varphi$ in order to ensure that the
  path is of the form
  $\varphi_1 \to \varphi_2 \leftarrow \varphi_3 \to \varphi_4
  \leftarrow \dots \to \varphi_n \leftarrow \varphi_1$. We will show
  that this loop is homotopic to
  $\delta(\varphi_1,\varphi_2,\varphi_3) \to
  \delta(\varphi_2,\varphi_3,\varphi_4) \leftarrow
  \delta(\varphi_3,\varphi_4,\varphi_5) \to
  \delta(\varphi_4,\varphi_5,\varphi_6) \leftarrow \dots \to
  \delta(\varphi_n,\varphi_1,\varphi_2) \leftarrow
  \delta(\varphi_1,\varphi_2,\varphi_3) $,
  see \Cref{fig:contract}.
  \begin{figure}
    \begin{center}
      \begin{tikzcd}
        ~& \\
        \varphi_1
        \arrow[dd]
        \arrow[u]
        \arrow[rd]&\\
        & \delta(\varphi_1,\varphi_2)
        \arrow[ld]&~\\
        \varphi_2
        &
        &\delta(\varphi_1,\varphi_2,\varphi_3)
        \arrow[ll]
        \arrow[dd]
        \arrow[ld]
        \arrow[lu]
        \arrow[u]\\
        & \delta(\varphi_2,\varphi_3)
        \arrow[lu]
        \arrow[rd]\\
        \varphi_3
        \arrow[dd]
        \arrow[uu]
        \arrow[ru]
        \arrow[rd]
        \arrow[rr]&
        &\delta(\varphi_2,\varphi_3,\varphi_4)\\
        &\delta(\varphi_3,\varphi_4)
        \arrow[ld]
        \arrow[ur]&\\\
        \varphi_4&&
        \delta(\varphi_3,\varphi_4,\varphi_5)
        \arrow[uu]
        \arrow[d]
        \arrow[ld]
        \arrow[lu]
        \arrow[ll]\\
        & \delta(\varphi_4,\varphi_5)
        \arrow[ld]
        \arrow[lu]&~\\
        \varphi_5
        \arrow[d]
        \arrow[uu]
        &
        &\\
        ~
      \end{tikzcd}.
    \end{center}
    \caption{Contracting the loop $\varphi_1 \rightarrow \varphi_2
      \leftarrow \varphi_3 \rightarrow \varphi_4 \leftarrow \varphi_5
      \dots$.}
    \label{fig:contract}
  \end{figure}
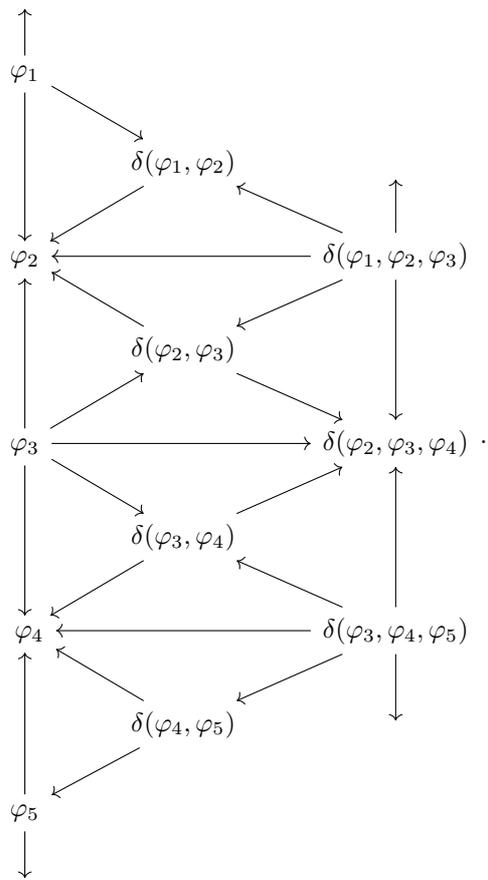
  Property \eqref{eq:conditions-arrow-2-H} and \eqref{eq:conditions-arrow-3-H}
  show that all expressions in the right column
  are defined.
  For the middle column property
  \eqref{eq:conditions-drop-H}
  is used additionally. For example we have $\varphi_1 \leftarrow
  \varphi_1 \rightarrow \varphi_2$, hence
  $(\varphi_1,\varphi_1,\varphi_2) \in H_3$, thus
  $(\varphi_1,\varphi_2) \in H_2$.

  The arrows from the left column to the middle column
  are implied by
  \eqref{eq:conditions-equivalent},
  \eqref{eq:conditions-arrow-delta},
  \eqref{eq:conditions-drop-delta},
  \eqref{eq:conditions-triv-H} and
  \eqref{eq:conditions-triv-delta}.
  For example $\varphi_1 \to \varphi_1$
  and $\varphi_1 \to \varphi_2$
  implies $\varphi_1 \cong \delta(\varphi_1)
  \cong \delta(\varphi_1, \varphi_1) \to \delta(\varphi_1,\varphi_2)$.
  
  For the arrows from the middle to the right column
  we additionally need \eqref{eq:conditions-grouping-H}
  and \eqref{eq:conditions-grouping-delta}.
  For example $\varphi_2 \leftarrow \delta(\varphi_2,\varphi_3)$
  and $\delta(\varphi_2,\varphi_3) \leftarrow
  \delta(\varphi_2,\varphi_3)$
  implies
  \begin{align*}
    &\delta(\varphi_2,\varphi_3)\cong\delta(\varphi_2,\varphi_2,\varphi_3) \cong
    \delta(\varphi_2,\delta(\varphi_2,\varphi_3)) \\
    \leftarrow&
    \delta(\delta(\varphi_1,\varphi_2),\delta(\varphi_2,\varphi_3))
    \cong \delta(\varphi_1,\varphi_2,\varphi_2,\varphi_3) \cong
    \delta(\varphi_1,\varphi_2,\varphi_3).
  \end{align*}
  The same properties together with \eqref{eq:conditions-permutation-H}
  and \eqref{eq:conditions-permutation-delta}
  also imply the arrows from
  the left column to the right column since
  $\varphi_2 \leftarrow \delta(\varphi_1,\varphi_2)$
  and $\varphi_2 \leftarrow \delta(\varphi_2,\varphi_3)$
  \begin{align*}
    \varphi_2 \cong \delta(\varphi_2,\varphi_2) \leftarrow
    \delta(\varphi_1,\varphi_2,\varphi_2,\varphi_3) \cong \delta(\varphi_1,\varphi_2,\varphi_3).
  \end{align*}
  
  Applying the same arguments again and using
  \eqref{eq:conditions-grouping-H}
  and 
  \eqref{eq:conditions-grouping-delta}
  we see that our loop is homotopic to
  \begin{align*}
    &\delta(\varphi_1,\varphi_2,\varphi_3,\varphi_4,\varphi_5) \to 
      \delta(\varphi_2,\varphi_3,\varphi_4,\varphi_5,\varphi_6) \leftarrow 
      \delta(\varphi_3,\varphi_4,\varphi_5,\varphi_6,\varphi_7) \to \\
    &\delta(\varphi_4,\varphi_5,\varphi_6,\varphi_7,\varphi_8) \leftarrow  
      \dots 
      \to
      \delta(\varphi_n,\varphi_1,\varphi_2,\varphi_3,\varphi_4) \leftarrow  
      \delta(\varphi_1,\varphi_2,\varphi_3,\varphi_4,\varphi_5).
  \end{align*}
  Continuing this procedure we obtain a homotopy to a path all of
  whose vertices are of the form
  $\delta(\varphi_{\kappa(1)}),\dots,\delta(\varphi_{\kappa(n)})$ for
  a permutation $\kappa$.  Such a loop can be contracted to
  the point
  $\delta(\varphi_1,\varphi_2,\varphi_3,\varphi_4,\dots,\varphi_n)$
  by \eqref{eq:conditions-permutation-delta}.
\end{proof}

\section{Markov shifts on finitely generated groups}
\label{sec:markov-shifts-fg}
Let $\CG$ be a finitely generated group and let $S \subseteq \CG$ be a
finite generating set containing the identity which does not have to be
symmetric. The most interesting case is actually when
$S \cap S^{-1} = \{e\}$ where $e$ is the identity element of the
group.

Let $K$ be a finite set of symbols
endowed with the discrete topology.
The space of configurations $X = K^\CG$ endowed with the product topology is a compact, totally
disconnected, metrizable space.
$\CG$ acts on $K^\CG$ from the left via
$\sigma_g(x)_h=x_{g^{-1}h}$.

We now consider topological Markov shifts of $X$ in the following sense.
\begin{defn}
Let $\CR \subseteq K^S$ be a set of patterns.
We call
\begin{align*}
  X_\CR:=\{x \in X \setsep (\sigma_g(x))_{|S} \in \CR \text{
  for all } g \in \CG  \}
\end{align*}
an \emph{$S$-Markov shift} with
\emph{allowed} patterns $\CR$.
\end{defn}
Let $X, Y$ be two $S$-Markov shifts over the alphabets $K$ and $L$. A
map $\varphi: X \to Y$ commuting with the $\CG$-action has
neighborhood $T \subseteq \CG$ if there is a map $\varphi_\loc: K^T \to K$
with $\varphi_g(x)=\varphi_\loc(\sigma_{g^{-1}}(x)_{|T})$.
Such a map is called a \emph{sliding block map}.

Now consider the groupoid $\Gamma$ of $S$-Markov shifts
over alphabets of the form $K=\{1,\dots,k\}$ for $k \in \N$ together with
invertible sliding block maps as morphisms.
By the Curtis-Lyndon-Hedlund theorem (see \cite{ceccherini-silbersteinCellularAutomataGroups2010a}), these are precisely the homeomorphisms commuting with the
$\CG$ action, hence we call them (topological) \emph{conjugacies}.
A conjugacy with neighborhood $\{e\}$
whose inverse also has neighborhood $\{e\}$ will be called
an \emph{alphabet bijection}.
Let $H \subseteq \Gamma^1$ be the set of all sliding block maps with
neighborhood $S$, whose inverse has neighborhood $S^{-1}$.
We call these conjugacies \emph{elementary}.

For an $S$-Markov shift $X$ let $\Gamma_X$ be the connected component of $X$ in $\Gamma$
and let $H_X := H \cap \Gamma_X^{1}$.
We now want to show that $H_X$ homotopically uniquely generates $\Gamma_X$.

Consider the following basic construction.
Let $\varphi_i: X \to Y_i$ be bijective sliding block maps for $1\leq i \leq n$.
We define 
\begin{align*}
	\varphi_1\star \dots \star \varphi_n: X \to Y_1 \times \dots \times Y_n,\quad x \mapsto (\varphi_1(x),\dots,\varphi_n(x)).
\end{align*}
This is a bijective map onto its image \[\im(\varphi_1\star \dots \star \varphi_n)=
\{(\varphi_1(x),\dots,\varphi_n(x)) \setsep x \in X\}.\]
To see recognize when this image is an $S$-Markov shift, the following
observation is useful.
For any $i \in \{1,\dots,n\}$ we can write the image as
\begin{align}
  &\im(\varphi_1\star \dots \star \varphi_n) \\
  &
    = \{(x_1,\dots,x_n) \in Y_1 \times \dots \times Y_n \setsep
    \varphi_j^{-1}(x_j)
    =\varphi_i^{-1}(x_i) \text{ for } j \in \{1,\dots,n\}\} \\
  &
    = \{(x_1,\dots,x_n) \in Y_1 \times \dots \times Y_n \setsep
    (\varphi_i\circ\varphi_j^{-1})(x_j)
    = x_i \text{ for } j \in \{1,\dots,n\} \} \label{eq:im-ij}\\
  &
    = \{(x_1,\dots,x_n) \in Y_1 \times \dots \times Y_n \setsep
    (\varphi_j\circ\varphi_i^{-1})(x_i)
    = x_j \text{ for } j \in \{1,\dots,n\} \}\label{eq:im-ji} 
\end{align}
Now let $H_n$ be the set of $n$-tuples of elementary conjugacies in $H_X$ for which
$\im(\varphi_1 \star \dots \star \varphi_n)$ is $S$-Markov.  To define
$\delta$ we have to ensure that the alphabet of our image is a
consecutive sequence of integers starting at one.
Define
\[L_1(\im(\varphi_1 \star \dots \star\varphi_n)) :=
  (\varphi_1(x)_0,\dots,\varphi_n(x)_0).\]
Notice that this set is
totally ordered by the lexicographic order, hence there is a canonical
choice of enumeration
\[r: L_1(\im(\varphi_1\star \dots\star \varphi_n)) \to \{1,\dots,
  |\im(\varphi_1\star \dots\star \varphi_n)|\}.\]
Then for
$\varphi_1,\dots,\varphi_n \in H_n$ the map
$\delta_n(\varphi_1,\dots,\varphi_n):= r \circ (\varphi_1 \star \dots \star
\varphi_n)$ is a sliding block map from $X$ to an $S$-Markov subshift
over the alphabet \[\{1,\dots, |\im(\varphi_1\star \dots \star \varphi_n)|\}.\]  Hence
$\delta_n$ maps $H_n$ to $\Gamma^1_X$. It will be useful to also define
$\delta_n(\varphi_1,\dots,\varphi_n)$ for such
$\varphi_1,\dots,\varphi_n \in H_X^{-1}$ for which
$\im(\varphi_1\star\dots\star\varphi_n)$
is $S$-Markov using the same definition as
above.

\begin{lem}
  \label{lem:refinment-ongroups}
  The triple $(\cong,(H_n)_{n \in \N},(\delta_{n \in \N}))$ as defined above is a refinement structure for
  $(\Gamma_X,H_X)$.
\end{lem}
\begin{proof}
We have to check properties \eqref{eq:conditions-equivalent} to \eqref{eq:conditions-arrow-delta}.
For \eqref{eq:conditions-equivalent} let $\varphi_1 \cong \varphi_2$,
$\varphi_3 \cong \varphi_4$
and $\varphi_1 \to \varphi_3$. There are alphabet bijections
$f$ and $g$ such that
$\varphi_1 = f \circ \varphi_2$ and $\varphi_4=g \circ \varphi_3$.
We have
\begin{align*}
\varphi_4 \circ \varphi_2^{-1}
&= \varphi_4 \circ \varphi_3^{-1} \circ
\varphi_3 \circ \varphi_1^{-1} \circ
\varphi_1 \circ \varphi_2^{-1}\\
&= g \circ \varphi_3 \circ \varphi_1^{-1} \circ f.
\end{align*}
Composition with a bijection of the alphabet from either side does not change the
radius of a block map. Hence $\varphi_2 \to \varphi_4$, which shows
\eqref{eq:conditions-equivalent}.

Conditions \ref{eq:conditions-triv-H} and
\ref{eq:conditions-triv-delta} are straightforward.

Permutation of the factors in $\varphi_1 \star \dots \star \varphi_n$
only changes the order of the coordinates and hence
generates the same subshift up to permutation of the alphabet. This
shows \eqref{eq:conditions-permutation-H}
and \eqref{eq:conditions-permutation-delta}.

By definition
\[\im(\delta_k(
  \delta_n(\varphi_1^1,\dots,\varphi_n^1),\dots,
  \delta_n(\varphi_1^k,\dots,\varphi_n^k)))\]
and
\[\im(\delta_{kn}(\varphi_1^1,\dots,\varphi_n^1,\dots,\varphi_1^k,\dots,\varphi_n^k))\]
are actually the same subshift, hence one of them is $S$-Markov if and
only if the other is.  This shows \eqref{eq:conditions-grouping-H} and
\eqref{eq:conditions-grouping-delta}.

The same argument applies to
$\delta(\varphi_1,\varphi_1,\dots,\varphi_n)$
and $\delta(\varphi_1,\dots,\varphi_n)$
which are the same subshift, showing
\eqref{eq:conditions-drop-H} and \eqref{eq:conditions-drop-delta}.

If $\varphi_1 \leftarrow \varphi_2 \rightarrow \varphi_3$ then
\eqref{eq:im-ij} implies $(\varphi_1,\varphi_2,\varphi_3) \in H_3$
and hence \eqref{eq:conditions-arrow-2-H}.
Simlarly for 
$\varphi_1 \rightarrow \varphi_2 \leftarrow \varphi_3$ then
\eqref{eq:im-ji} implies $(\varphi_1,\varphi_2,\varphi_3) \in H_3$
and hence \eqref{eq:conditions-arrow-3-H}.

Finally $\varphi_1 \to \varphi_2$ and $\varphi_3 \to \varphi_4$ implies
$\delta(\varphi_1, \varphi_3) \to \delta(\varphi_2,\varphi_4)$.
Notice that for $(y,z) \in \im (\varphi_1 \star \varphi_3)$ we have
$\delta(\varphi_1, \varphi_3)^{-1}(y,z)=\varphi_1^{-1}(y)=\varphi_3^{-1}(z)$
and hence
\begin{align*}
  \delta(\varphi_2,\varphi_4)(\delta(\varphi_1,\varphi_3)^{-1}(y,z))
  &=\varphi_2(\delta(\varphi_1, \varphi_3)^{-1}(y,z)),
    \varphi_4(\delta(\varphi_1, \varphi_3)^{-1}(y,z)) \\
  &=\varphi_2(\varphi_1^{-1}(y)), \varphi_4(\varphi_3^{-1}(z)).
\end{align*}
Hence $\delta(\varphi_1,\varphi_3) \to \delta(\varphi_2,\varphi_4)$,
thus \eqref{eq:conditions-arrow-delta}.
\end{proof}

Now we only have to show that $H_X \cup H_X^{-1}$ actually generates
$\Gamma_X$.  In the case $\CG=\Z$ this is the classical argument of
Williams as presented for example in \cite[Theorem
7.1.2]{lindIntroductionSymbolicDynamics1995}.  In the general case we need
slightly more notation. Set $S'=S \cup S^{-1}$.  Then $S'$ is a
symmetric generating set of $\CG$ and we can consider the Cayley graph
of $\CG$ with respect to this generating set.  We say that
$T \subseteq \CG$ is \emph{$S'$-connected} if $T$ induces a connected
subgraph in this Cayley graph. For an $S$-Markov shift $Y$ and
$g \in S$ let $\tau_{g,Y} \in H$ be the map defined by
$\tau_{g,Y}(y)_h = y_{hg}$.
 
 \begin{lem}
   \label{lem:smaller-neighborhood}
	Let $T \subseteq G$ be finite and $S'$-connected with $T \setminus \{e\} \neq \emptyset$.
	Then there exists $g \in S \cup S^{-1}$ and $h \in T \setminus \{e\}$ such that 
	$T':=(T \setminus \{h\})$ is $S'$-connected and
	$T \subseteq T' \cup T'g$.
\end{lem} 
\begin{proof}
  Pick a spanning tree of the subgraph induced by $T$ of the Cayley
  graph of $\CG$ with respect to $S'$. Let $h \neq e$ be a leaf in this
  tree.  Removing $h$ from $T$ clearly leaves it $S'$-connected and
  there is $h' \in T \setminus \{h\}, g \in S'$ such that $h=
  h'g$. Hence $T':= T \setminus \{h\}$ is $S'$-connected and
  $T \subseteq T' \cup T'g$.
\end{proof} 

\begin{lem}
  \label{lem:higherblock}
  Let $\varphi: X \to Y$ be a sliding block map with $S'$-connected
  neighborhood $T \neq \{e\}$.  There is a map $\psi$ in
  $H \cup H^{-1}$ such that $\varphi \circ \psi$ has neighborhood $T'$
  with $|T'|=|T|-1$ where $T'$ is again $S'$-connected.
\end{lem}
\begin{proof}
	Let $h \in T,g \in S \cup S^{-1}$ and $T':=T \setminus \{h\}$ be as in \Cref{lem:smaller-neighborhood}. 
	In particular there is $h' \in T'$ with $h'g=h$.
   By \eqref{eq:im-ij} and \eqref{eq:im-ji} 
   $\im(\id_X \star \tau_{g,X})$ is $S$-Markov and we can define
   \[\psi:=\delta(\id_X,\tau_{g,X})^{-1} = (\id_X \star
     \tau_{g,X})^{-1} \circ r^{-1}.\]
   Notice that $\delta(\id_X, \tau_{g,X})$ is either in $H$ or $H^{-1}$,
   depending on whether $g \in S$ or $g \in S^{-1}$.
	We want to show that $\varphi \circ \psi$ has
   neighborhood $T'$.
	Since $r$ is just a permutation of the alphabet, it is enough to show that
	$\varphi \circ (\id_X \star \tau_{g,X})^{-1}$ has neighborhood $T'$.
	Let $(x,y) \in \im(\id_{X} \star\, \tau_{g,X})$. Then $y=\tau_{g,X}(x)$,
	hence $x_{h}=x_{h'g}=\tau_{g,X}(x)_{h'}=y_{h'}$. Therefore
	$(x,y)_{|T'}$ determines $x_{|T}$, thus $\psi \circ (\id \star \tau_{g,X})^{-1}$
	has neighborhood $T'$.
\end{proof}

\begin{lem}
  \label{lem:generate-mainstep}
  Let $\varphi: X \to Y$ be a one block map whose inverse has an
  $S'$-connected neighborhood $T \neq \{1_G\}$. There are maps
  $\psi_1,\psi_2$ in $H \cup H^{-1}$ such that
  $\psi_2 \circ \varphi \circ \psi_1^{-1}$ is a one block map whose
  inverse has an $S'$-connected neighborhood $T'$ with $|T'|=|T|-1$.
\end{lem}
\begin{proof}
  Again let
  $h \in T,g \in S \cup S^{-1}$ and $T':=T \setminus \{h\}$ be as in
  \Cref{lem:higherblock},
  so there is $h' \in T'$ with $h'g=h$.  Define
	\begin{align*}
	\psi_1'&:=\id_X \star (\tau_{g,Y} \circ \varphi), &
	\psi_2'&:=\id_Y \star\, \tau_{g,Y}, \\
	\psi_1&:=r\circ \psi_1'=\delta(\id_X,\tau_{g,Y} \circ \varphi), &
	\psi_2&:=r\circ \psi_2'=\delta(\id_Y,\tau_{g,Y}).
	\end{align*}
   The following commutative diagram summarizes the situation.
	\begin{center}
		\begin{tikzcd}[column sep=3.2cm]
			X \arrow[r, "\varphi"] \arrow[d, "\id_X \star\, (\tau_{g,Y} \circ \varphi)" left] 
			&Y \arrow[d, "\id_Y \star\, \tau_{g,Y}"] 
			\\
			\im(\id_X \star\, (\tau_{g,Y} \circ \varphi)) 
			\arrow[r, "\psi_2' \circ \varphi  \circ {\psi_1'}^{-1}" below]
			\arrow[d,"r" left]
			& \im(\id_Y \star\, \tau_{g,Y}) \arrow[d,"r"]
			\\
			\im(\delta(\id_X, \tau_{g,Y} \circ \varphi)) 
			\arrow[r, "\psi_2 \circ \varphi  \circ {\psi_1}^{-1}" below]	
			& \im(\delta(\id_Y, \tau_{g,Y}))
		\end{tikzcd}
	\end{center}
	Now $\id_X \star\, (\tau_{g,Y} \circ \varphi)$ has neighborhood $\{e,g\}$ and its
   inverse has neighborhood $\{e\}$.
	Hence it is contained in $H \cup H^{-1}$.
	The same holds true for $\id_Y \star\, \tau_{g,Y}$.
	Finally we have for $(x,x') \in \im(\id_X \star\, (\tau_{g,Y} \circ \varphi))$
	\begin{align*}
		(\psi'_2 \circ \varphi \circ {\psi'_1}^{-1}) (x,x') 
		&= (\id_Y \star\, \tau_{g,Y} )(\varphi(x)) \\
		&=(\varphi(x),\tau_{g,Y}(\varphi(x)) \\
		&=(\varphi(x),x').
	\end{align*}
	Hence $\psi'_2 \circ \varphi \circ {\psi'_1}^{-1}$ is a one-block map.
	For  $(y,y') \in \im(\id_Y \star\, \tau_{g,Y})$ we have
	\begin{align*}
		(\psi'_1 \circ \varphi^{-1} \circ {\psi'_2}^{-1}) (y,y') 
		&= (\id_Y \star\, (\tau_{g,Y} \circ \varphi))(\varphi^{-1}(y)) \\
		&= (\varphi^{-1}(y),\tau_{g,Y}(y)) \\
		&= (\varphi^{-1}(y),y').
	\end{align*}	
	Now $y_{h}=y_{h'g}=(\tau_{g,Y}(y))_{h'}=y'_{h'}$, hence $y_{|T}$ is
   uniquely determined by $(y,y')_{|T'}$ and therefore
   $\psi'_1 \circ \varphi^{-1} \circ {\psi'_2}^{-1}$ has neighborhood
   $T'$.  Since $r$ is merely a bijection of alphabets, the map
   $\psi_2 \circ \varphi \circ {\psi_1}^{-1}$ is also a one-block map whose
   inverse has neighborhood $T'$.
\end{proof}

\begin{thm}
  \label{thm:hug-over-groups}
  The groupoid $\Gamma_X$ is homotopically uniquely generated by $H_X$.
\end{thm}
\begin{proof}
  We first show that $H_X$ indeed generates $\Gamma_X$.
  For this it is enough to show that every
  conjugacy $\varphi: X  \to Y$ for a $S$-Markov shift $Y$
  is a concatenation of elementary conjugacies in $H \cup H^{-1}$.
  By repetitively applying \eqref{lem:higherblock} we can precompose $\varphi$
  with conjugacies from $H \cup H^{-1}$ such that
  the resulting map is a one-block map from
  some subshift $\tilde{X} \in \Gamma_X^0$ to $Y$.
  We may thus assume that $\varphi$ is already a one-block map.
  
  Now applying \Cref{lem:generate-mainstep} repetitively
  we can pre- and postcompose $\varphi$ by elements
  of $H \cup H^{-1}$ 
  keeping it a one-block map
  and decreasing the size of the neighborhood
  of the inverse by one in each step.
  We end up with a one-block map whose
  inverse is also a one-block map,
  in other words, it is an alphabet bijection and thus contained in
  $H$.
  This shows that $\varphi$ can be represented
  as the composition of elements in $H$ and $H^{-1}$.
  Homotopically unique generation is now a direct consequence of
  \Cref{lem:refinment-ongroups} and 
  \Cref{thm:simply-connected}.
\end{proof}

\begin{cor}\label{cor:fundamenatl-S-Markov}
  For every $S$-Markov shift $X \subseteq K^\CG$ the automorphism group $\Aut(X)$ is isomorphic
  to the fundamental group $\pi_1(\NN(\Gamma_X,H_X),X)$.
\end{cor}


\section{The Matrix Equations}
\label{sec:matrix-eq}
For $\CG=\Z$ there is a natural choice for a generating set containing
the identity, namely $S=\{0,1\}$. For this choice of
generating set $S$-Markov shifts correspond
to the classical topological Markov shifts as discussed in
\Cref{sec:strong-shift-eq} and the set
of elementary conjugacies discussed there corresponds
to the elements of $H$ defined in \Cref{sec:markov-shifts-fg}.
We saw that every such elementary conjugacy can
be described by a pair of matrices $(R,S)$.
If we manage to express the fact that $\varphi_1\circ
\varphi_2=\varphi_3$
in terms of the corresponding matrices, we
can get a presentation of the complex of elementary conjugacies in
purely algebraic terms. This is the goal of this section.

We start by actually proving the correspondence
between elementary conjugacies and elementary strong shift equivalences.

\begin{thm}\label{thm:elementary-from-algebraic}
	Let $X_A$ and $X_B$ be Markov shifts defined by non-degenerate matrices 
	$A \in \{0,1\}^{m \times m}$ and $B \in \{0,1\}^{n \times n}$.
	If $R \in \{0,1\}^{m \times n}$ and $S \in \{0,1\}^{n \times m}$
	are matrices such that $A=RS$ and $B=SR$
   then there is a uniquely defined elementary conjugacy $\varphi_{R,S}: X_A \to X_B$ with
	\begin{align}
		(\varphi_{R,S})_{\loc}(a,a')=b &\iff R_{a,b}S_{b,a'}=1, \label{eq:RS-1}\\
		(\varphi^{-1}_{R,S})_{\loc}(b,b')=a &\iff S_{b,a}R_{a,b'}=1.\label{eq:RS-2} 
	\end{align}
\end{thm}

\begin{proof}
  Since $A$ and $B$ are non-degenerate $\{0,1\}$-matrices with $A=RS$ and
  $B=SR$, for each $a,a' \in \{1,\dots,m\}$ with $A_{a,a'}=1$ and $b,b' \in
  \{1,\dots,n\}$ with $B_{b,b'}=1$ there are uniquely defined elements
  $c \in \{1,\dots,n\}$
  and $d \in \{1,\dots,m\}$ with $R_{a,c}S_{c,a'}=1$ and $S_{b,d}S_{d,b'}=1$.
  Hence we can define two maps with neighborhood $\{0,1\}$ and  $\{-1,0\}$, respectively,
  by
  \begin{align}
    	\varphi_\loc(a,a')=c &\iff R_{a,c}S_{c,a'}=1, \label{eq:RS-1}\\
		\psi_\loc(b,b')=d &\iff S_{b,d}R_{d,b'}=1.\label{eq:RS-2}.
  \end{align}
  Let $x \in X_A$ and $y:=\varphi(x)$.
  To prove our theorem, it is enough to show that $x=\psi(y)$.
  We have
  $R_{x_{i}y_{i}}S_{y_{i},x_{i+1}}=1=R_{x_{i-1},y_{i-1}}S_{y_{i-1},x_i}$,
  hence $S_{y_{i-1},x_{i}}R_{x_i,y_i}$ and thus
  $\psi(y)_i=\psi_\loc(y_{i-1},y_i)=x_i$.
  Hence $\varphi\circ \psi = \id$. That $\psi \circ \varphi=\id$
  follows by symmetry.
\end{proof}

\begin{thm}\label{thm:algebraic-from-elementary}
  Let $X_A$ and $X_B$ be Markov shifts defined by non-degenerate
  $\{0,1\}$-matrices $A \in \{0,1\}^{m \times m}$ and
  $B \in \{0,1\}^{n \times n}$ and let $\varphi: X_A \to X_B$ be an
  elementary conjugacy. Define a pair of matrices
  $R_\varphi \in \{0,1\}^{m \times n}$ and
  $S_\varphi \in \{0,1\}^{n \times m}$ as follows.
	\begin{align*}
		(R_\varphi)_{a,b}=1 &\iff \exists a' \in \{1,\dots,m\}: \varphi_{\loc}(aa')=b, \\
		(S_\varphi)_{b,a}=1 &\iff \exists b' \in \{1,\dots,n\}: \varphi^{-1}_{\loc}(bb')=a.
	\end{align*}
	Then $R_\varphi S_\varphi =A$ and $S_\varphi R_\varphi=B$.
\end{thm}

Before we come to the proof we state and prove a little lemma.
\begin{lem}\label{lem:R-characterization}
  Let $X_A, X_B, \varphi, R_\varphi$ as in the statement of
  \Cref{thm:algebraic-from-elementary}.  For
  $a \in \{1,\dots,m\}, b \in \{1,\dots,n\}$ the following are equivalent.
	\begin{enumerate}[(i)]
		\item\label{enum:R-chara-1} $(R_\varphi)_{a,b}=1$,
		\item\label{enum:R-chara-2} $\exists a' \in \{1,\dots,m\}: \varphi_{\loc}(a,a')=b$,
		\item\label{enum:R-chara-3} $\exists b' \in \{1,\dots,n\}: \varphi^{-1}_{\loc}(b',b)=a$,
		\item\label{enum:R-chara-4} $\exists x \in X_A, y\in X_B: x_0=a, y_0=b, y=\varphi(x)$ .
	\end{enumerate}
\end{lem}
\begin{proof}
	By definition properties \eqref{enum:R-chara-1} and \eqref{enum:R-chara-2} are equivalent.
	We now show the equivalence of   \eqref{enum:R-chara-2} and  \eqref{enum:R-chara-4}.
	Property \eqref{enum:R-chara-4} immediately implies \eqref{enum:R-chara-2}
	as we can choose $a':=x_1$.
	On the other hand if $a,a'$ and $b$ are such that $\varphi_\loc(a,a')=b$, then 
	we can find a point $x \in X_A$ with $x_0=a, x_1=a'$. Setting $y:=\varphi(x)$
	gives  \eqref{enum:R-chara-4}.
	Finally, \eqref{enum:R-chara-4} is equivalent to the existence of $x \in X_A, y\in X_B,
	x_0=a, y_0=b$ and $x =\varphi^{-1}(x)$. As we just saw, this is equivalent to
	\eqref{enum:R-chara-3}.
\end{proof}

\begin{proof}[Proof of \Cref{thm:algebraic-from-elementary}]
	We will show that $(R_{\varphi})_{a,b}(S_{\varphi})_{b,c}=1$
	iff $A_{a,c}=1$ and $b=\varphi_{\loc}(a,c)$.
	This will directly imply $R_\varphi S_\varphi=A$,
	since for $A_{a,c}=1$ we have 
	\begin{align*}
     (R_{\varphi}S_{\varphi})_{a,c}
     = \sum_{b \in \{1,\dots,n\}}  (R_{\varphi})_{a,b}(S_{\varphi})_{b,c}
     = \sum_{b=\varphi_{\loc}(ac)} (R_{\varphi})_{a,b}(S_{\varphi})_{b,c}
     =1
	\end{align*}
	and for $A_{a,c}=0$ we have
	\begin{align*}
     (R_{\varphi}S_{\varphi})_{a,c}
     =\sum_{b \in \{1,\dots,n\}} (R_{\varphi})_{a,b}(S_{\varphi})_{b,c}
     =0.
	\end{align*}
	
	First let $b=\varphi_{\loc}(a,c)$.
   There is $x \in X_A$ with $x_0=a,x_1=c$ and $\varphi(x)_0=b$.
   Hence for $b':=\varphi(x)_{1}$ we have $\varphi_\loc^{-1}(b,b')=c$,
   showing $(R_\varphi)_{a,b}(S_\varphi)_{b,c}=1$.
   
	On the other hand assume that
   $(R_{\varphi})_{a,b}(S_{\varphi})_{b,c}=1$
   for some $b$.
   We can find $\tilde{x} \in X_A$ and $\tilde{y}' \in X_B$
   such that $\tilde{x}_0=a, \varphi(\tilde{x})_0=b=\tilde{y}'_0$
   and $\varphi^{-1}(\tilde{y}')_1=c$.
   Set $\tilde{y}:=\varphi(\tilde{x})$ and $\tilde{x}':=\varphi^{-1}(\tilde{y})$.
    Define $y \in X_B$
   by
   \begin{align*}
     y_{i} = \begin{cases}
       \tilde{y}_i& \text{for } i\leq 0 \\
       \tilde{y}'_i& \text{for } i> 0
     \end{cases}
   \end{align*}
   and $x:=\varphi^{-1}(y)$.
   Then $x_0=\varphi^{-1}(\tilde{y})=a,
   x_1=\varphi^{-1}(\tilde{y}')=c$
   and $y_0=b$, hence $b=\varphi_\loc(a,c)$.

   Now if $\varphi$ is an elementary conjugacy, so
   is $\varphi^{-1} \circ \sigma$ as
   \begin{align*}
     (\varphi^{-1}\circ \sigma)(y)_i&=\sigma(\varphi^{-1}(y))_i
                                      =\varphi^{-1}(y)_{i+1} \\
                              &=\varphi_\loc^{-1}(y_i,y_{i+1}),  \\
     (\varphi^{-1}\circ \sigma)^{-1}(y)_i&=\sigma^{-1}(\varphi(y))_i=\varphi(y)_{i-1} \\
                              &=\varphi_\loc(y_{i-1},y_i).     
   \end{align*}
   It is also easy to see that $R_{\sigma \circ \varphi^{-1}} =
   S_{\varphi}$
   and $S_{\sigma \circ \varphi^{-1}} =
   R_{\varphi}$.
	Hence $B=R_{\sigma \circ \varphi^{-1}}S_{\sigma \circ \varphi^{-1}}=S_\varphi R_\varphi$.
\end{proof}

We now capture commuting triangles in the complex
of elementary conjugacies by three triangle equations
for the corresponding matrix pairs.
\begin{thm}
  \label{thm:triangle-equations}
	Let the following be a (not necessarily commuting) triangle of elementary 
	conjugacies.
	\begin{center}
		\begin{tikzcd}
			& X_B \arrow[rd,"\varphi_2"] &\\
			X_A \arrow[ru,"\varphi_1"] \arrow[rr,"\varphi_3" below]& &X_C 
		\end{tikzcd}
	\end{center}
	Set $R_i=R_{\varphi_i}$ and $S_i=S_{\varphi_i}$ for $i=1,2,3$.
	Then the triangle 
	commutes if and only if the following three \emph{triangle equations} hold
	\begin{align*}
	R_1R_2&=R_3, \\
	R_2S_3&=S_1, \\
	S_3R_1&=S_2.
	\end{align*}.
\end{thm}


\begin{proof}
  Assume the triangle commutes.
	Let $a,b,c$ be symbols such that $(R_1)_{a,b}(R_2)_{b,c}=1$.
   By \Cref{lem:R-characterization} we know that
   there are $\tilde{x} \in X_A, \tilde{y} \in X_B, \tilde{y}' \in Y_B$ and $\tilde{z} \in X_C$
   with $\tilde{x}_0=a,\tilde{y}_0=\tilde{y}'_0=b$ and $\tilde{z}_0=c$ such that
   $\tilde{y}=\varphi_1(\tilde{x})$ and $\tilde{z}=\varphi_2(\tilde{y}')$.
   Define $y \in X_B$
   by
   \begin{align*}
     y_{i} = \begin{cases}
       \tilde{y}_i& \text{for } i\leq 0 \\
       \tilde{y}'_i& \text{for } i> 0
     \end{cases}.
   \end{align*}.
   Set $x:=\varphi_1^{-1}(\tilde{y})$ and
   $\tilde{z}:=\varphi_2(\tilde{y})$.
   We thus found configurations $x,y,z$ with
   \begin{align*}
     x_0&=\varphi_1^{-1}(\tilde{y})_0=a, \\
	  y_0&=\tilde{y}_0=b, \\
	  z_0&=\varphi_2(\tilde{y}')_0=c, \\
	  y&=\varphi_1(x), z=\varphi_2(y).
   \end{align*}
   This implies in particular that $(R_3)_{a,c}=1$.

   Now assume there is another symbol $b'$ with
   $(R_1)_{a,b'}(R_2)_{b',c}$.
   As we saw, there must be elements $x'',y'',z''$ with
   $x''_0=a,y''_0=b', z''_0=c$ such that
   $y''=\varphi_1(x'')$ and $z''=\varphi_2(y'')$.
   Define $\hat{x} \in X_A$ and $\hat{z} \in X_C$ by
   \begin{align*}
     \hat{x}_i &= \begin{cases}
       x_i &\text{for } i\leq 0 \\
       x''_i &\text{for } i >0
     \end{cases},
     &
     \hat{z}_i &= \begin{cases}
       z_i &\text{for } i\leq 0 \\
       z''_i &\text{for } i >0
     \end{cases}.
   \end{align*}
   Now $\varphi_3(\hat{x})=\hat{z}$ and hence
   $b=\varphi_1(\tilde{x})_0=\varphi_2^{-1}(\hat{z})_0=\tilde{b}$.
   This shows $(R_3)_{a,c} \geq (R_1R_2)_{a,c}$ for all
   pairs of symbols $a,c$.
   To see the converse inequality consider $a,c$
   with $(R_3)_{a,c}=1$. By \Cref{lem:R-characterization}
   there is $x,z$ with $\varphi_3(x)=z$ and $x_0=a,z_0=c$.
   Define $y:=\varphi_1(x)$. Then
   $\varphi_2(y)=\varphi_2(\varphi_1(x))=z$
   and hence $(R_1)_{a,y_0}(R_2)_{y_0,b}=1$.
   All in all this shows $R_1R_2=R_3$.
   Recall that for every elementary conjugacy $\varphi$,
   so conjugacy $\varphi^{-1} \circ \sigma$ is also elementary.
   Hence, from the commuting triangle we started with we get two other
   commuting triangles.
   \begin{center}
		\begin{tikzcd}
        & X_B \arrow[rd,leftarrow,"\sigma \circ\varphi_2^{-1}"] &\\
        X_A \arrow[ru,"\varphi_1"] \arrow[rr,leftarrow,"\sigma \circ
        \varphi_3^{-1}" below]& &X_C
		\end{tikzcd}
      \hspace{1cm}
      \begin{tikzcd}
        & X_B \arrow[rd,"\varphi_2"] &\\
        X_A \arrow[ru,leftarrow,"\sigma \circ \varphi_1^{-1}"]
        \arrow[rr,leftarrow,"\sigma \circ \varphi_3^{-1}" below]& &X_C
		\end{tikzcd}
    \end{center}
    Calculating the first triangle equation for these
    triangles we obtain
    \begin{align*}
      R_{\sigma \circ \varphi_3^{-1}} R_{\varphi_1}
      &= S_3 R_1 = R_{\sigma \circ \varphi_2^{-1}} = S_2, \\
      R_{\varphi_2} R_{\sigma \circ \varphi_3^{-1}}
      &= R_2 S_3 = R_{\sigma \circ \varphi_1^{-1}} = S_1.
    \end{align*}
    
	Now assume the triangle equations hold.
	Let $x \in X_A$, $y:= \varphi_1(x)$, $z:=\varphi_2(y)$, $w:=\varphi_3(x)$.
	We have to show that $x=w$. Because $x$ is arbitrary and all maps under consideration are
	shift invariant it is enough to show $x_0=w_0$.
	Since $w_0$ is the unique index such that $(R_3)_{x_0 w_0} (S_3)_{w_0 x_1}=1$,
	we have to show that $(R_3)_{x_0 z_0} (S_3)_{z_0 x_1}=1$.
	We know that 
	\begin{align*}
		(R_3)_{x_0 z_0} &\geq  (R_1)_{x_0 y_0} (R_2)_{y_0 z_0} = 1.
	\end{align*}
	Since all matrices we consider have entries in $\{0,1\}$, this implies
	$(R_3)_{x_0 z_0}=1$.
	Since $1 = (S_1)_{y_0 x_1}= (R_2S_3)_{y_0 x_1}$,
	there is a unique index $i$ such that
	$1 = (R_2)_{y_0 i}(S_3)_{i x_1}$. Multiplication by $(R_1)_{x_1 y_1}$ on the right 
	gives \[1=(R_2)_{y_0 i}(S_3)_{i x_1}(R_1)_{x_1 y_1}=(R_2)_{y_0 i}(S_2)_{i y_1}.\]
	Hence $i=z_0$ and $(R_3)_{x_0 z_0} (S_3)_{z_0 x_1}=1$ as required. \qedhere
	
\end{proof}

In light of these results we construct the following $\Delta$-complex in purely algebraic terms.
\begin{defn}\label{def:SSE}
	Let $Q$ be a subset of some semiring.
	Let $A \subseteq Q^{n \times n}$. Define
	$\SSE(Q)$ as the simplicial set whose vertices are non-degenerate square matrices over
	$K$ and whose $n$-simplices are $n+1$-tuples of such matrices $(B_0,\dots,B_n)$
	together with pairs of non-degenerate matrices $R^{i,j},S^{i,j}$ over $Q$
	such that $B_i=R^{i,j}S^{i,j}$ and $B_j=S^{i,j}R^{i,j}$.
	Let $\SSE(Q,A)$ be the connected component of $\SSE(Q)$ containing $A$.	
\end{defn}

\begin{thm}\label{thm:simplicial-set-iso}
  Let $\Lambda$ be the groupoid of Markov subshifts, and let
  $E$ be the set of elementary conjugacies. Then
	the map $\NN(\Lambda,E) \to \SSE(\{0,1\})$ induced by
	\begin{align*}
	X_A &\mapsto A, \\
	\varphi &\mapsto (R_\varphi, S_\varphi)
	\end{align*} 
	is an isomorphism of simplicial sets.
   To be precise,
   under this map a simplex given by elementary conjugacies $\varphi_{i,j}:X_{B_i}
   \to X_{B_j}$
   is mapped to
   the simplex defined by the elementary
   strong shift equivalences $(R_{\varphi_{i,j}},S_{\varphi_{i,j}})$.
\end{thm}
\begin{proof}
  That the map maps simplices to simplices is the content of
  \Cref{thm:triangle-equations}.
  Bijectivity follows from the observation that $R_{\varphi_{R,S}}=R, S_{\varphi_{R,S}}=S$ 
  and that $\varphi_{R_\varphi,S_\varphi}=\varphi$.
\end{proof}
\begin{cor}
  \label{cor:sse-automorphismgroup}
  Let $A \in \{0,1\}^{n \times n}$ be non-degenerate.
  Then $\Aut(X_A)$ is isomorphic to $\pi_1(\SSE(\{0,1\}),A)$.
\end{cor}
\begin{proof}
  This is a direct consequence of \Cref{cor:fundamenatl-S-Markov} and \Cref{thm:simplicial-set-iso}.
\end{proof}

\section{$G$-SFTs}
\label{sec:gsfts}
In this section we deal with subshifts of finite type on which a
finite group $G$ acts freely by automorphisms, a so called \emph{$G$-SFT} (see
for example \cite{boyleFlowEquivalenceGSFTs2015a}
or \cite{boyleFiniteGroupExtensions2017}).
Such systems can be described by square matrices over
the subset $R^\star \subseteq \Z G$ 
of the integer group ring of all elements whose
coefficients are in $\{0,1\}$.
Our goal is again to show that the group
of automorphisms of the system described by such a matrix $A$
is equal to the fundamental group of $\SSE(R^\star)$
based at $A$.

In order to assign a uniquely defined matrix over $R^\star$ to a $G$-SFT
first pick a fixed total order on $G$.
Let $D$ be a finite directed graph without sinks or sources and
without parallel edges on which $G$ acts freely.
Mark one vertex in every orbit
and choose a total order on the set $M$ of these marked vertices.
We call the resulting structure a \emph{marked $G$-graph}.
The vertex set of every marked $G$-graph can be relabeled by $\{1,\dots, |M|\} \times G$
as follows. 
The order on $M$ defines a unique enumeration $r: M \to
\{1,\dots,|M|\}$.
Since the action of $G$ is free, every vertex is of the form $gm$ for
a unique $g \in G$, $m \in M$. Label this vertex by $(r(m),g)$.

We therefore can always assume that the
vertex set of a marked $G$-graph is of the form $\{1,\dots,n\} \times G$
and that $G$ acts via $g(k,h)=(k,gh)$.

Now we describe how to get such a marked $G$-graph from
a matrix $A \in (G^\star)^{n \times n}$.
First we associate to $A$ a matrix $\overline{A}$
indexed by $(\{1,\dots,n\} \times G)\times (\{1,\dots,n\} \times G)$
with entries in $\{0,1\}$ as follows.
The entry $\overline{A}_{(k,g),(\ell,h)}$ equals $1$
if and only if the coefficient of  $g^{-1}h$ in $A_{k,\ell}$
is non-zero.

Let $e$ be the identity element of $G$.
Let $E$ be a matrix indexed by $(\{1,\dots,m\} \times G) \times
(\{1,\dots,n\} \times G)$ with entries in $\{0,1\}$
such that $E_{(k,e),(\ell,h)}=E_{(k,g),(\ell,gh)}$
for every $g,h \in G$.
We can define a matrix $\hat{E}$ over $G^\star$
of size $m \times n$
via 
\begin{align*}
  \hat{E}_{k,\ell} = \sum_{h \in G} E_{(k,e),(\ell,h)} h. 
\end{align*}
The operations $A \mapsto \overline{A}$
and $E \mapsto \hat{E}$ are inverse to each other.
The following calculation shows that these operations are
multiplicative.
Let $A \in (G^\star)^{m \times n}$, 
$B \in (G^{\star})^{n \times k}$ be matrices over $G^\star$.
Then 
\begin{align*}
  (AB)_{x_1,x_2}
  &= \sum_{y_1} A_{x_1,y_1} B_{y_1,x_2} \\
  &= \sum_{y_1} \sum_{g}\sum_{h}
    \overline{A}_{(x_1,e),(y_1,g)} \overline{B}_{(y_1,e),(x_2,h)} gh\\
  &= \sum_{y_1}\sum_{g}\sum_{h}
    \overline{A}_{(x_1,e),(y_1,g)} \overline{B}_{(y_1,g),(x_2,gh)} gh\\
  &= \sum_{y_1}\sum_{g}\sum_{h}
    \overline{A}_{(x_1,e),(y_1,g)} \overline{B}_{(y_1,g),(x_2,h)} h\\
  &= \sum_{h}
    (\overline{A}\cdot\overline{B})_{(x_1,e),(x_2,h)} h \\
  &=\widehat{\overline{A}\cdot\overline{B}}_{x_1,x_2}
\end{align*}
and hence $\overline{AB}=\overline{A}\cdot \overline{B}$
and $\widehat{AB}=\widehat{A}\widehat{B}$.

For a square matrix $A$ over $G^\star$ let $D_A$ be the graph with
adjacency matrix $\hat{A}$, i.e.  its vertex set is
$\{1,\dots,n\} \times G$ and there is an edge from $(k_1,g_1)$ to
$(k_2,g_2)$ if $g_2$ has non-zero coefficient in $g_1A_{k_1,k_2}$.
Then $G$ acts freely on this graph from the left via $h(k,g)=(k,hg)$.
Denote by $X_A$ the vertex shift of $D_A$.

\begin{exam}
  Consider the group $G:=\Z/3\Z$ written
  multiplicatively with generator $a$ .
  $G$ acts on the following graph freely as described above.
  
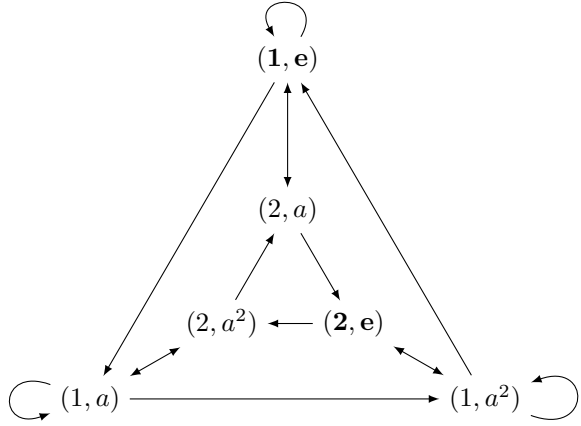
\begin{center}
    \begin{tikzpicture}
      \node (1e) at ( 90:3.0) {$\mathbf{(1,e)}$};
      \node (1a) at (210:3.0) {$(1,a)$};
      \node (1b) at (330:3.0) {$(1,a^2)$};
      \node (2e) at (330:1.0) {$\mathbf{(2,e)}$};
      \node (2a) at ( 90:1.0) {$(2,a)$};
      \node (2b) at (210:1.0) {$(2,a^2)$};
      \begin{scope}[-latex]
        \draw (1e)--(1a);
        \draw (1a)--(1b);
        \draw (1b)--(1e);
        \draw (2a)--(2e);
        \draw (2b)--(2a);
        \draw (2e)--(2b);
      \end{scope}
      \begin{scope}[latex-latex]
        \draw (1e)--(2a);
        \draw (1a)--(2b);
        \draw (1b)--(2e);
      \end{scope}
  \begin{scope}[-latex]
        \draw (1e) to [out= 60, in=120, looseness=5] (1e);
        \draw (1a) to [out=160, in=200, looseness=5] (1a);
        \draw (1b) to [out=340, in=380, looseness=5] (1b);
      \end{scope}
    \end{tikzpicture}
  \end{center}
  This system is described by the matrix
  \begin{align*}
    A&=\begin{pmatrix}
      e+a & a \\
      a^2 & a^2
    \end{pmatrix}.
  \end{align*}
  Ordering $G$ as $(e,a,a^2)$ we obtain
  the following matrix as $\overline{A}$ which indeed is the
  adjacency matrix of the graph we started with.
  \begin{align*}
    \overline{A} &=\begin{pmatrix}
             1 & 1 & 0 & 0 & 1 & 0 \\
             0 & 1 & 1 & 0 & 0 & 1 \\
             1 & 0 & 1 & 1 & 0 & 0 \\
             0 & 0 & 1 & 0 & 1 & 0 \\
             1 & 0 & 0 & 0 & 0 & 1 \\
             0 & 1 & 0 & 1 & 0 & 0
    \end{pmatrix}
  \end{align*}
\end{exam}

\begin{defn}
Let $\Lambda$ be the groupoid whose objects are non-degenerate
square matrices over $G^\star$ and
whose morphisms from $A$ to $B$ are
conjugacies between $X_A$ and $X_B$ commuting with the
action of $G$.
\end{defn}
With these notations in place, we can state our main theorem in this setting. 
\begin{thm}
  The fundamental groupoid $\pi_1(\SSE(G^\star))$ is isomorphic
  to $\Lambda$ via
  \begin{align*}
    A &\mapsto A \\
    ((R_1,S_1)^{\eps_1}, \dots ,(R_n,S_n)^{\eps_n}) &\mapsto
    \varphi_{R_n,S_n}^{\eps_n} \circ \dots \circ \varphi_{R_1,S_1}^{\eps_1}
  \end{align*}
\end{thm}
\begin{proof}
Let
the set of generators $E$ be the elementary conjugacies
commuting with the $G$ action.

As before, $\varphi_1: X_A \to X_B$
and $\varphi_2: X_A \to X_C$
are equivalent if $\varphi_2 \circ \varphi_1^{-1}$
is an alphabet bijection. This equivalence relation fulfills
\eqref{eq:conditions-equivalent} as was shown in \Cref{sec:markov-shifts-fg}.

The sets $H_n$ are defined as before.
To define $\delta$ let $\varphi_i: X_A \to X_{B_i}$ be elementary
conjugacies.  Consider $\varphi_1 \star \dots \star \varphi_n$ mapping
$x \in X_A$ to
$(\varphi_1(x),\dots,\varphi_n(x)) \in X_{B_1} \times \dots \times
X_{B_n}$.  If
$\im(\varphi_1 \star \dots \star \varphi_n)$ is a Markov shift, it is
the vertex shift of the graph with vertex set
$\{\varphi_1(x)_0,\dots,\varphi_n(x)_0 \setsep x \in X_A\}$ and edges
from $(\varphi_1(x)_0,\dots,\varphi_n(x)_0)$ to
$(\varphi_1(x)_1,\dots,\varphi_n(x)_1)$ for all
$x \in X_A$.
The group $G$ acts naturally on
this graph via
\begin{align*}
  g(\varphi_1(x)_0,\dots,\varphi_n(x)_0)&=(\varphi_1(gx)_0,\dots,\varphi_n(gx)_0) \\
                                        &=(g\varphi_1(x)_0,\dots,g\varphi_n(x)_0)
\end{align*}
and this action is free since the action of $G$
on the graph defining $X_{B_1}$ is free.

The only thing left to define $\delta(\varphi_1,\dots,\varphi_n)$
is to choose the set $M$ of distinguished vertices in this graph.
Every orbit contains exactly one vertex of the form
$(y_1,e),(y_2,g_2),\dots,(y_n,g_n)$.
Let $M$ be the set of these vertices.
This uniquely determines a matrix $C$ of size $|M|\times |M|$ over
$G^{\star}$. Let $\delta(\varphi_1,\dots,\varphi_n)$
be the conjugacy defined by $\varphi_1\star\dots\star\varphi_n$
from $X_A \to X_C$.
As before \eqref{eq:conditions-triv-H} and
\eqref{eq:conditions-triv-delta} are
immediately fulfilled.
For 
\eqref{eq:conditions-grouping-H}
notice that
\[\im(\delta_k(
  \delta_n(\varphi_1^1,\dots,\varphi_n^1),\dots,
  \delta_n(\varphi_1^k,\dots,\varphi_n^k)))\]
and
\[\im(\delta_{kn}(\varphi_1^1,\dots,\varphi_n^1,\dots,\varphi_1^k,\dots,\varphi_n^k))\]
are again the same subshift, hence one of them is $S$-Markov if and
only if the other is. We also
marked the same set of symbols for both of these
subshifts so the maps
\[\delta_k(
  \delta_n(\varphi_1^1,\dots,\varphi_n^1),\dots,
  \delta_n(\varphi_1^k,\dots,\varphi_n^k))\]
and
\[\delta_k(
  \delta_n(\varphi_1^1,\dots,\varphi_n^1),\dots,
  \delta_n(\varphi_1^k,\dots,\varphi_n^k))\]
  are in fact equal, showing \eqref{eq:conditions-grouping-delta}.
The same is true for
$\delta_{n+1}(\varphi_1,\varphi_1,\dots,\varphi_n)$
and $\delta_n(\varphi_1,\dots,\varphi_n)$.
proving \eqref{eq:conditions-drop-H}
\eqref{eq:conditions-drop-delta}

For \eqref{eq:conditions-permutation-H} and
\eqref{eq:conditions-permutation-delta}
notice that
$\delta(\varphi_{\kappa(1)},\dots,\varphi_{\kappa(n)})
\delta(\varphi_{1},\dots,\varphi_{n})^{-1}$
is given by the alphabet permutation
mapping $((y_1,e),\dots,(y_n,g_n))$
to \[((y_{\kappa(1)},e),\dots,(y_{\kappa(n)},g_{\kappa(1)}^{-1}g_{\kappa(n)})).\]

Properties
\eqref{eq:conditions-arrow-2-H},
\eqref{eq:conditions-arrow-3-H}
and
\eqref{eq:conditions-arrow-delta}
have already been shown to hold in \Cref{sec:markov-shifts-fg}.
All in all this shows that we have a refinement structure
and we can apply \Cref{thm:simply-connected}.

Each elementary conjugacy between the vertex shifts defined by
$A\in {G^{\star}}^{m\times m}$ and $B \in {G^{\star}}^{n \times n}$
can be described by a pair of matrices $\tilde{R}, \tilde{S}$
over $\{0,1\}$ such that $\tilde{R}\tilde{S}=\overline{A}$ and
$\tilde{S}\tilde{R}=\overline{B}$.  Since our elementary conjugacies commute with the
$G$-action, for every $g \in G$ we have
\begin{align*}
  S_{(y_1,e),(x_2,g)}=S_{(y_1,h),(x_2,hg)} \text{ and} \\
  R_{(x_1,e),(y_1,g)}=S_{(x_1,h),(y_1,hg)}
\end{align*}

The matrices $\overline{R}$ and $\overline{S}$ give rise to
matrices $R=\hat{\tilde{R}}$ and $S=\hat{\tilde{S}}$.
We have $A=\hat{\overline{A}}=\widehat{\tilde{R}\tilde{S}}=RS$.
and similarly $B=SR$.
Conversely every such pair of matrices $R,S$ produces
an elementary conjugacy commuting with the $G$ action
by first passing to $\overline{R}, \overline{S}$.

The multiplicativity of $A\mapsto \overline{A}$ and $E \mapsto
\hat{E}$ also show that
the triangle equations hold for the $R,S$ matrices
if and only if they hold for the corresponding $\overline{R},\overline{S}$
matrices. Hence we can finish the proof by applying
to \Cref{thm:triangle-equations} as in the proof
of \Cref{thm:simplicial-set-iso}.
\end{proof}

\section{Degenerate Matrices}
\label{sec:degenerate}

The one-to-one correspondence between square $\{0,1\}$-matrices and topological
Markov shifts only works with non-degenerate matrices
or equivalently directed graphs without sinks and sources
as all adjacency information of sinks and sources is lost
when we pass to two-sided infinite paths.
Nevertheless one can define for a subset $Q$
of a semiring the space $\SSE_{\text{deg}}(Q)$
using \Cref{def:SSE} but
allowing matrices with zero rows or columns.

This gives additional algebraic flexibility 
that Wagoner and Boyle needed in
\cite{zbMATH05347058} to establish
a link between strong shift equivalence
and positive row and column operations
of polynomial matrix representations, thus bring methods of K-theory to the table. 
For $Q \subseteq \Z_{\geq 0}$ they showed \cite[Theorem A.7]{zbMATH05347058} that
every pair of non-degenerate matrices
that is connected by a path in $\SSE_{\text{deg}}(Q)$
is also connected by a path in $\SSE(Q)$.
They also showed \cite[Proposition A.11]{zbMATH05347058} that every degenerate matrix
is connected to a non-degenerate matrix in $\SSE_{\text{deg}}(Q)$.
Expressing this in terms of path components this means
that the inclusion
$\SSE(Q) \to \SSE_{\text{deg}}(Q)$
induces an isomorphism
\begin{align*}
  \pi_0(\SSE(Q)) \to \pi_0(\SSE_{\text{deg}}(Q)).
\end{align*}
Boyle and Wagoner asked \cite[Question 1]{zbMATH05347058} if the same holds for the fundamental groups, i.e.,
is there also an induced isomorphism
\begin{align*}
  \pi_1(\SSE(Q),A) \to \pi_1(\SSE_{\text{deg}}(Q),A)
\end{align*}
for non-degenerate $A$?

The following theorem shows that this indeed the case.
\begin{thm}\label{thm:iso-deg-fundamental}
  For $Q \subseteq \Z_{\geq 0}$ the inclusion
  $\SSE(Q) \to \SSE_{\text{deg}}(Q)$ induces an isomorphism
  of fundamental groups
  \begin{align*}
    \pi_1(\SSE(Q),A) \to \pi_1(\SSE_{\text{deg}}(Q),A)
  \end{align*}
  for every non degenerate square matrix $A$ over $Q$.
\end{thm}

The main ingredient of the proof is the following lemma for which
we have to introduce some notation.
For a matrix $C \in \Z^{n \times m}$ and non-empty index sets $K \subseteq \{1,\dots,n\}=:N$
and $L \subseteq  \{1,\dots,m\}=:M$ let $C_{K\times L}$ be the
matrix obtained from $C$ by removing all rows not in $K$ and all
columns not in $L$.
Let $I_{K \times L}$ be the corresponding submatrix
of a sufficiently large identity matrix.
For a square matrix $A \in \Z_{\geq 0}^{m \times m}$
denote by $J_A$ the set of all indices $i$ such that
$A^k_{i,\cdot} \neq (0,\dots,0)$ and $A^k_{\cdot,i} \neq (0,\dots,0)$
for all $k$. These are precisely the indices that can appear
in two-sided infinite paths of the graph with adjacency matrix $A$ .

\begin{lem}\label{lem:deg-triangulate}
Consider the  matrices $A \in \Z_{\geq 0}^{n \times n}$,
$B \in \Z_{\geq 0}^{m \times m}$, $R \in \Z_{\geq 0}^{n \times m}$,
and $S \in \Z_{\geq 0}^{m \times n}$ with $A=RS, B=SR$.
Let $K \subseteq \{1,\dots,n\}$ be the set of all indices $k$ for
which $A_{k,\cdot}\neq (0,\dots,0)$ and let
$L \subseteq \{1,\dots,m\}$ be the set of all indices $\ell$ for which
$B_{\ell,\cdot}\neq (0,\dots,0)$. 
There are matrices $R',S'$ such that the diagram
\begin{center}	
	\begin{tikzcd}[column sep=2.2cm]
		A_{K\times K} 
		\arrow[r,"{R',S'}"]
		&B_{L\times L}
		\\
		A 
		\arrow[u,"{I_{N \times K},A_{K \times N}}"]
		\arrow[r,"{R,S}" below]
		&B
		\arrow[u, "{I_{M \times L}, B_{L \times M}}" right]
      
	\end{tikzcd}
\end{center}
can be triangulated by four triangles in $\SSE_\text{deg}(Q)$.
\end{lem}

\begin{proof}
Define a diagonal matrix $E_S \in \{0,1\}^{m \times m}$ by 
\[(E_S)_{i,i}=\begin{cases}1 &\text{if } \exists k \in \{1,\dots,n\} \text{ with } S_{i,k}=1 \\
0&\text{otherwise} \end{cases}.\]
 Denote by $J \subseteq M$ the set
of all indices $j$ for which $S_{j,\cdot} \neq (0,\dots,0)$.  Since
$B=SR$, we clearly have $M \setminus J \subseteq M \setminus L$, hence
$L \subseteq J$.

Our diagram can now be triangulated as follows.
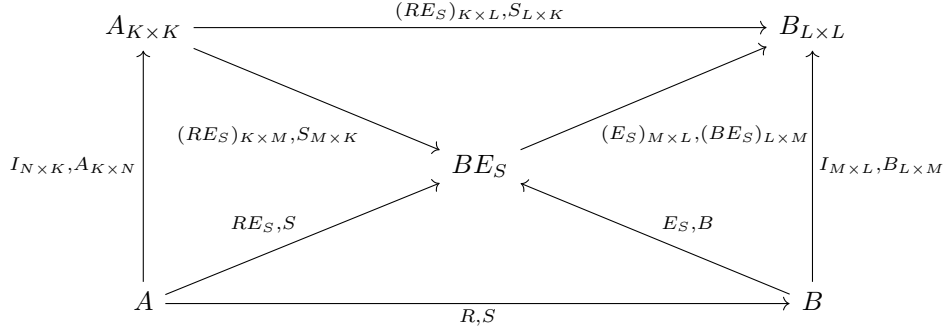
\begin{center}	
	\begin{tikzcd}[column sep=3.25cm, row sep=huge]
		A_{K\times K} 
		\arrow[rr,"{(RE_S)_{K\times L},S_{L \times K}}"]
		\arrow[rd,"{(RE_S)_{K\times M},S_{M \times K}}" below left, pos=0.7]
		&&B_{L\times L}
		\\
		&B E_S
		\arrow[ru,"{(E_S)_{M\times L},(BE_S)_{L \times M}}" below right , pos=0.3]
		& \\
		A 
		\arrow[uu,"{I_{N \times K},A_{K \times N}}"]
		\arrow[ur,"{RE_S,S}" above left]
		\arrow[rr,"{R,S}" below]
		&&B
		\arrow[uu, "{I_{M \times L}, B_{L \times M}}" right]
		\arrow[lu, "{E_S,B}" above right]
	\end{tikzcd}
\end{center}
Before we start checking the triangle equations, we make some simple observations.
Multiplying a matrix by $E_S$ from the left 
sets all rows in $M \setminus J$ to zero and
leaves all rows in $J$ unchanged. Therefore $E_SS=S$
and $E_SB=E_SSR=SR=B$.
Two other equations are central, namely
\begin{align}
  S_{M \times K}(RE_S)_{K \times M}
  &= BE_S, \label{eq:SRE=BE} \\
  (RE_S)_{K \times L}S_{L \times K}
  &= A_{K \times K}. \label{eq:RES=A}
\end{align}

For the first equation, consider $S_{i,k}(RE_S)_{k,j}$.
If \[S_{i,k}(RE_S)_{k,j}=S_{i,k}R_{k,j}(E_S)_{j,j}>0\] then $j \in J$, hence there
must be $\ell \in N$ with $S_{j,\ell} > 0$ and thus
$A_{k,\ell} \geq R_{k,j}S_{j,\ell} > 0$. Therefore $k \in K$.
This shows $S_{M \times K}(RE_S)_{K \times M} = SRE_S=BE_S$.

For the second equation, consider $(RE_S)_{k,i}S_{i,\ell}$ with $\ell \in K$.
There is $j \in M$ and $\ell'$ with $A_{\ell,j}\geq R_{\ell,\ell'}S_{\ell',j}>0$.
Now if $(RE_S)_{k,i}S_{i,\ell} = R_{k,i} (E_S)_{i,i} S_{i,\ell} > 0$,
we have $0<S_{i,\ell}R_{\ell,\ell'} \leq B_{i,\ell'}$,
hence $i \in K$. Thus
$(RE_S)_{K \times L}S_{L \times K} = (RE_S S)_{K \times K} = A_{K
  \times K}$.

Now we have to check that all matrix pairs along the edges are
indeed elementary strong shift equivalences between the matrices at
their source and target. This amounts to checking that
\begin{align}
  A&=RS,
  &B&=SR, \\
  A&=RE_SS,
  &SRE_S &= BE_S, \\
  B&=E_SB,
  &BE_S &= BE_S, \\
  (E_S)_{M \times L}(BE_S)_{L \times M}&=BE_S, 
  &(BE_S)_{L \times M}(E_S)_{M \times L}&=B_{L \times L}, \\
  (RE_S)_{K \times M}S_{M \times K}&=A_{K \times K}, 
  &S_{M \times K}(RE_S)_{K \times M}&=BE_S, \\
  I_{N \times K}A_{K \times N} &= A, 
  &A_{K \times N}I_{N \times K} &= A_{K \times K}, \\
  I_{M \times L}B_{L \times M} &= B, 
  &B_{L \times M}I_{M \times L} &= B_{L \times L}, \\                        
  (RE_S)_{K \times L}S_{L \times K} &= A_{K \times K},
  & S_{L \times K} (RE_S)_{K \times L} &= B_{L \times L.}
\end{align}
All of these equations are either direct consequences of the
definitions of $L$ and $K$ or follow directly from \eqref{eq:SRE=BE} and
\eqref{eq:RES=A}.

Now consider the first triangle
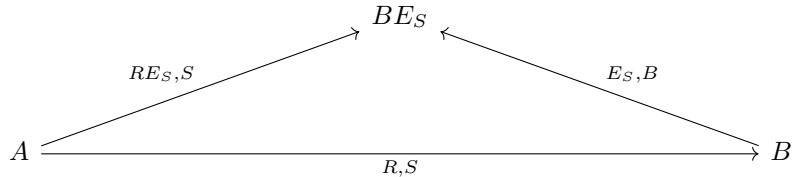
\begin{center}	
	\begin{tikzcd}[column sep=4.2cm, row sep=huge]
		&BE_S
		& \\	
		A 
		\arrow[ur,"{RE_S,S}" above left]
		\arrow[rr,"{R,S}" below]
		&&B		
		\arrow[lu, "{E_S,B}" above right]
	\end{tikzcd}
\end{center}
Here we have to check
\begin{align}
R \cdot E_S&= RE_S,&(R_1R_2&=R_3,) \\
SR&=B,	&(S_3 R_1, &= S_2) \\
E_S S &=S.  &(R_2 S_3 &= S_1)	
\end{align}
All of these equations have already been shown.

We continue with the second triangle:
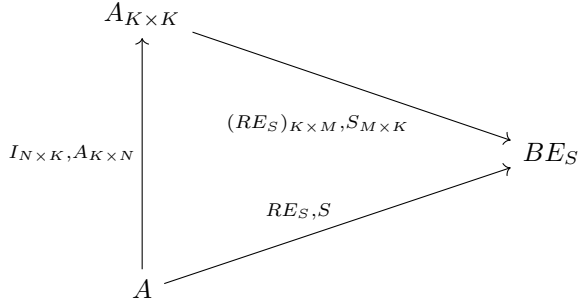
\begin{center}	
	\begin{tikzcd}[column sep=4.2cm, row sep=huge]
		A_{K\times K} 
		\arrow[rd,"{(RE_S)_{K\times M},S_{M \times K}}" below left, pos=0.7]
		&&
		\\
		&B E_S
		& \\
		A 
		\arrow[uu,"{I_{N \times K},A_{K \times N}}"]
		\arrow[ur,"{RE_S,S}" above left]
		&&
	\end{tikzcd}
\end{center}
Here we have to check that
\begin{align*}
I_{N \times K} (RE_S)_{K \times M}, &= RE_S,&(R_1R_2&=R_3) \\
SI_{N \times K}&=S_{M \times K},	&(S_3 R_1 &= S_2) \\
(RE_S)_{K \times M} S &=A_{K \times N}.  &(R_2 S_3 &= S_1)	
\end{align*}
Only the first equation is non trivial.
Let $(RE_S)_{i,k}=R_{i,k}(E_S)_{k,k}\geq 1$ for some $i \in N, k \in M$.
Hence there is $j \in N$ with $S_{k,j}\geq 1$ and therefore
$A_{i,j}\geq R_{i,k}S_{k,j}\geq 1$. This implies $i \in K$.
Hence $(RE_S)_{i,k}=0$ for all $i \in N \setminus K, k \in M$
and thus $I_{N \times K} (RE_S)_{K \times M} = RE_S$.

For the third triangle
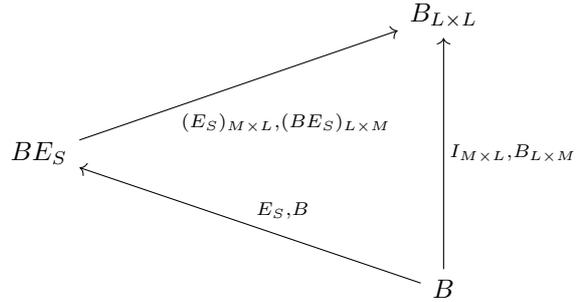
\begin{center}	
	\begin{tikzcd}[column sep=4.2cm, row sep=huge]
		&&B_{L\times L}
		\\
		&B E_S
		\arrow[ru,"{(E_S)_{M\times L},(BE_S)_{L \times M}}" below right , pos=0.3]
		& \\
		&&B
		\arrow[uu, "{I_{M \times L}, B_{L \times M}}" right]
		\arrow[lu, "{E_S,B}" above right]
	\end{tikzcd}
\end{center}
we have to check
\begin{align*}
E_S (E_S)_{M \times L} &= I_{M \times L},&(R_1R_2&=R_3) \\
B_{L \times M}E_S&=(BE_S)_{L \times M},	&(S_3 R_1 &= S_2) \\
I_{M \times L}B_{L\times M} &=B.  &(R_2 S_3 &= S_1)	
\end{align*}
For the first equality we have to check that $(E_S)_{M \times L}=I_{M \times L}$.
This follow directly from $L \subseteq H$. and $(E_S)_{M \times H}=I_{M \times H}$.
We also have $B_{i,m}=0$ for all $i \in N \setminus L, m \in M$ and therefore 
$I_{M \times L}B_{L\times M} =B$. The second equation is again trivial.

Finally we consider the triangle on the top.
\begin{center}	
	\begin{tikzcd}[column sep=4.2cm, row sep=huge]
		A_{K\times K} 
		\arrow[rr,"{(RE_S)_{K\times L},S_{L \times K}}"]
		\arrow[rd,"{(RE_S)_{K\times M},S_{M \times K}}" below left, pos=0.7]
		&&B_{L\times L}
		\\
		&B E_S
		\arrow[ru,"{(E_S)_{M\times L},(BE_S)_{L \times M}}" below right , pos=0.3]
		& 
	\end{tikzcd}
\end{center}
\begin{align*}
(RE_S)_{K \times M} (E_S)_{M \times L} &= (RE_S)_{K \times L}&(R_1R_2&=R_3) \\
S_{L \times K}(RE_S)_{K \times M}&=(BE_S)_{L \times M}	&(S_3 R_1 &= S_2) \\
(E_S)_{M \times L}S_{L \times K} &=S_{M \times K}  &(R_2 S_3 &= S_1)	
\end{align*}
Here the first equation is trivial and the second equation follow from
$(RE_S)_{i,m}=0$ for $i \in N \setminus K, m \in M$ which we already showed.
Finally consider $i \in N, k \in K$ with $S_{i,k}=1$.
There is $n \in N$ with $A_{k,n}\geq 1$ hence 
there is $j \in M$ with $R_{k,j}\geq 1$.
Therefore $B_{i,j}\geq S_{i,k}R_{k,j}\geq 1$. Thus $i \in L$.
This means $S_{i,k}=1$ for $i \in M \setminus L, k \in K$. 
Together with $(E_S)_{M \times L}=I_{M \times L}$ this shows 
$(E_S)_{M \times L}S_{L \times K}=S_{M \times K}$.
\end{proof}

\begin{lem}\label{lem:lift-nondeg}
  Let $Q \subset \Z_{\geq 0}$
  Consider the  matrices $A \in \Z_{\geq 0}^{m \times m}$,
  $B \in \Z_{\geq 0}^{n \times n}$, $C \in \Z_{\geq 0}^{p \times p}$
  and matrices $R_1,S_1,R_2,S_2,R_3,S_3$
  such that the following triangle fulfills the
  triangle equations.
  \begin{center}
	\begin{tikzcd}[column sep=2cm, row sep=large]
     &B
      \arrow[rd, "{R_2,S_2}" above right]
      & \\	
		A
		\arrow[ur,"{R_1,S_1}" above left]
		\arrow[rr,"{R_3,S_3}" below]
		&&C	
	\end{tikzcd}
 \end{center}
  Then the following triangle also fulfills the triangle equations.
    \begin{center}
	\begin{tikzcd}[column sep=2cm, row sep=large]
		&B_{J_B \times J_B}
      \arrow[rd, "{(R_2)_{J_B \times J_C},(S_2)_{J_C \times J_B}}" above right]
      & \\	
		A_{J_A \times J_A} 
		\arrow[ur,"{(R_1)_{J_A \times J_B},(S_1)_{J_B \times J_A}}" above left]
		\arrow[rr,"{(R_3)_{J_A \times J_C},(S_3)_{J_C \times J_A}}" below]
		&&C_{J_C \times J_C}
    \end{tikzcd}
 \end{center}
\end{lem}
\begin{proof}
  First we check that the edges in the second triangle indeed describe
  an elementary strong shift equivalence.
  To shorten notation write $R=R_1$ and $S=S_1$.
  Assume $R_{x_0y_0}S_{y_0x_1}>0$ for some $x_0, x_1 \in J_A$ and
  $y_0 \in \{1,\dots,n\}$. We can extend the word $x_0x_1$ to a two-sided
  infinite sequence $(x_k)_{k \in \Z}$ with $A_{x_k,x_{k+1}}>0$
  for all $k \in \Z$.
  For arbitrary $\ell \in \N$ we have
  \begin{align*}
   0&< A^\ell_{x_{-\ell},x_0} R_{x_0,y_0} S_{y_0,x_1}
      A^\ell_{x_1,x_{\ell+1}}
  \end{align*}
  and hence there are $y_{-\ell}, y_\ell \in \{1,\dots,m\}$ with
  \begin{align*}
  0&< R_{x_{-\ell},y_{-\ell}} B_{y_{-\ell},y_0}^\ell B_{y_0,y_\ell}^\ell
     S_{y_\ell,x_{\ell+1}}.
  \end{align*}
  Therefore $y_0 \in I_B$ and $A_{I_A \times I_A} = (R_1)_{I_A \times
    I_B}(S_1)_{I_B \times I_A}$.
  In the same way we see $B_{I_B \times I_B} = (S_1)_{I_B \times I_A}
  (R_1)_{I_A \times I_B}$.

  The proof for the triangle equations is very similar.
  We only show
  \[(R_1)_{I_A \times I_B}
  (R_2)_{I_B \times I_C} =
  (R_3)_{I_A \times I_C}.\]
  Let $x_0, y_0, z_0$ be indices such that
  $(R_1)_{x_0,y_0}(R_2)_{y_0,z_0}=(R_3)_{x_0,z_0}>0$.
  Extend $x_0$ and $z_0$ to bi-infinite sequenzes $x$ and $z$
  with $A_{x_k,x_{k+1}}>0$ and $C_{z_k,z_{k+1}}>0$ for all $k \in \Z$.
  For arbitrary $\ell \in \N$ we have
  \begin{align*}
    0<&A^\ell_{x_{-\ell},x_0} (R_1)_{x_0,y_0}, \\
    0<&(R_2)_{y_0,z_0}C^\ell_{z_0,z_{\ell}}.                      
  \end{align*}
  Hence there are $y_{-\ell},y_{\ell}$ such that
  \begin{align*}             
    0<& (R_1)_{x_{-\ell},y_{-\ell}} B^\ell_{y_{-\ell},y_0},\\
    0<& B^\ell_{y_0,y_\ell} (R_2)_{y_\ell,z_{\ell}}.
  \end{align*}
  Therefore $y_0 \in I_B$ and $(R_1)_{I_A \times I_B} (R_2)_{I_B
  \times I_C} = (R_3)_{I_A \times I_C}$.
\end{proof}

\begin{proof}[Proof of \Cref{thm:iso-deg-fundamental}]
  We first show injectivity.
  If two paths are homotopic in $\SSE_{\text{deg}}(Q)$
  then we can lift this homotopy by \Cref{lem:lift-nondeg}
  to $\SSE(Q)$ preserving endpoints.
  In particular for every non-degnerate square matrix $A$
  every contractible loop in
  $\SSE_{\text{deg}}(Q,A)$ is also contractible in $\SSE(Q,A)$.
  
  For surjectivity consider
  a path from $A$ to $B$ in $\SSE(Q)$.
  \Cref{lem:deg-triangulate} allows
  to homotop this path to a path
  containing only matrices without zero rows keeping endpoints fixed.
  Applying the lemma again and again we arrive
  at a path where all matrices have no zero rows
  also for all powers.
  
  Applying the lemma to the transposed matrices
  further homotops this path to one
  without zero columns.
  Doing this a finite number of times we finally arrive
  at a path in $\SSE(Q)$ homotopic to
  the original one with endpoints being fixed.

  
\end{proof}

\section{Contractability of $\PP(\Gamma,H,X)$}
\label{sec:contractability}

The aim of this section is to prove that
under additional assumptions
on the refinement structure
the coverings space $P:=\PP(\Gamma,H,X)$ is contractible.  The argument
is that of Wagoner in \cite[Proof of Proposition 2.12 Step
III]{wagonerMarkovPartitionsK21987} adapted to our notation
and with various details added. While this in principle
provides information about the group homology of
automorphism groups of subshifts of finite type, as far as we know
no concrete applications of the result are known so far. 

\begin{thm}
\label{thm:contractible}
  Let $\Gamma$ be a groupoid and let $H$ be a countable generating set
  containing all identities. If $(\Gamma,H)$ has a refinement
  structure $(\cong, (H_n)_{n \in \N}, (\pro_n)_{n \in \N})$ such that
  for all morphisms $\varphi_1,\varphi_2$
  $\varphi_1 \to \varphi_2$ together with $\varphi_2 \to \varphi_1$
  implies $\varphi_1 \cong \varphi_2$, then
  $\PP(\Gamma,H,X)$ is contractible for all objects $X$ of $\Gamma$. 
\end{thm}

Before we come to the proof of this theorem, which will make up the
rest of this section, we state a corollary for the automorphism
group of SFTs over $\Z$ or equivalently the space $\SSE(\{0,1\},A)$.

\begin{cor}
  \label{cor:classifying-space}
  Let $A$ be non-degenerate square $\{0,1\}$-matrix $A$. The space
  $\SSE(\{0,1\},A)$ is a classifying space of $\Aut(X_A)$. 
\end{cor}
\begin{proof}
  Let $H$ be the set of all conjugacies of Markov shifts over $\Z$ having neighborhood
  $\{0,1\}$ whose inverse has neighborhood $\{-1,0\}$. 
  We have to show that $H \cap H^{-1}$ consists only of alphabet
  permutations, because then the assumptions to
  \Cref{thm:contractible}
  are fulfilled. Assume that $\varphi \in H \cap H^{-1}$ is not
  an alphabet permutation.
  Let $\varphi_{\loc}$ be the local function of $\varphi$ for the
  neighborhood $\{0,1\}$.
  Assume that $\varphi$ does not have neighborhood $\{0\}$.
  Then there must be $a,b,c$ with
  $\varphi_{\loc}(a,b)\neq \varphi_{\loc}(a,c)$.
  Using the fact that $A$ is non-degenerate, we can extend the words
  $ab$ and $ac$ on both sides, generating configuration
  $x_1,x_2$ with $\varphi(x_1)_0 \neq \varphi(x_2)_0$ and
  $(x_1)_n=(x_2)_n$ for all $n \leq 0$.
  But this contradicts $\varphi$ having neighborhood $\{-1,0\}$,
  The same argument shows that $\varphi^{-1}$ has
  neighborhood $0$. Hence $\varphi$ is an alphabet permutation.
\end{proof}
\begin{rem}
  \label{rem:higher-dimension-H-H-1}
  The problem in higher dimensions is the fact that
  for the generating set $H$, as defined in \Cref{sec:markov-shifts-fg},
  $H \cap H^{-1}$ might contain more elements then just alphabet
  permutations.
  For example consider the $S$-Markov shift over $\Z^2$
  with 11 symbols \[\square, \uparrow, \downarrow, \leftarrow, \rightarrow,
  \nwarrow, \swarrow, \searrow, \nearrow, 1, 2\]
  and neighborhood $S=\{(1,0),(0,0),(0,1)\}$. where every
  $S$ pattern not appearing in \Cref{fig:subshift}
  is forbidden. As in \Cref{sec:markov-shifts-fg}
  let $H$ be the set of conjugacies having neighborhood $S$ and
  whose inverse has neighborhood $S^{-1}$.

  Consider the automorphism $\varphi$
  that exchanges $1$ and $2$ if they are surrounded by arrows.
  This is clearly not a alphabet permutation
  but it is in $H \cap H^{-1}$ as the knowledge
  of the value at a site $i$ and any of its neighbors
  determines the value of $\varphi$ at site $i$.
  \begin{figure}
\begin{center}
    \begin{tikzpicture}[scale=0.6]
\foreach \i in {-1,...,8}
{
	\draw (\i,-1) -- (\i,8);	
	\draw (-1,\i) -- (8,\i);	
}
\draw node at (1.5,1.5) {1};
\draw node at (0.5,1.5) {$\uparrow$};
\draw node at (2.5,1.5) {$\downarrow$};
\draw node at (1.5,2.5) {$\rightarrow$};
\draw node at (1.5,0.5) {$\leftarrow$};
\draw node at (2.5,0.5) {$\swarrow$};
\draw node at (0.5,0.5) {$\nwarrow$};
\draw node at (0.5,2.5) {$\nearrow$};
\draw node at (2.5,2.5) {$\searrow$};
\draw node at (5.5,1.5) {1};
\draw node at (1.5,5.5) {2};
\draw node at (5.5,5.5) {2};
\draw node at (4.5,5.5) {$\uparrow$};
\draw node at (6.5,5.5) {$\downarrow$};
\draw node at (5.5,6.5) {$\rightarrow$};
\draw node at (5.5,4.5) {$\leftarrow$};
\draw node at (6.5,4.5) {$\swarrow$};
\draw node at (4.5,4.5) {$\nwarrow$};
\draw node at (4.5,6.5) {$\nearrow$};
\draw node at (6.5,6.5) {$\searrow$};
\end{tikzpicture}
\end{center}
\caption{A subshift with an exotic automorphism in $H \cap H^{-1}$}
\label{fig:subshift}
\end{figure}

  Now one could also just define $\cong$ as
  the smallest equivalence relation containing $\to$, but then
  it is not clear if one can find $\delta$ to get a refinement structure.
\end{rem}

For the homological calculations, which we
will need for the proof of \Cref{thm:contractible}, it is easier to work with the space of unordered Markov
partitions as Wagoner does.

Assume we have a refinement structure $(\cong,(H_n)_{n \in \N},
(\delta_n)_{n \in \N}$ such that $\cong$) is the smallest equivalence
relation containing $\to$.
Pick
a representative from each equivalence class of $\cong$ and let $r(\varphi)$ be the
representative of the equivalence class of $\varphi$.
Denote the concatenation of $\delta$ followed
by $r$ by $\tilde{\delta}$.
Let $U$ be the maximal subcomplex of $P$ whose vertices are
these representatives.

\begin{lem}
  \label{lem:deformation-retract}
  The subcomplex $U$ is a deformation retract of $P$.
\end{lem}
\begin{proof}
  We are going to define a homotopy $\Phi_t$
  from $P$ to $U$ fixing $U$.
  Our goal is to map a point 
  $p = \sum_{i=1}^{n} \alpha_i \varphi_i$
  in  $[\varphi_1,\dots,\varphi_n]$ by $\Phi_t$ to
  $\sum_{i=1} (1-t) \alpha_i \varphi_i + t \alpha_i r(\varphi_i)$.
  We only have to specify in which simplex this image should lie.
  Since we  want $\Phi_t$ to be constant on $U$,
  the right simplex to pick is the one where
  we add $r(\varphi_i)$ after every occurrence of
  $\varphi_i$ with $\varphi_i \neq r(\varphi_i)$.
  For example if $\varphi_2=r(\varphi_2)$,
  $\varphi_1 \neq r(\varphi_1)$ and $\varphi_3 \neq r(\varphi_3)$
  then $\Phi_t([\varphi_1,\varphi_2,\varphi_3]) \subseteq
  [\varphi_1,r(\varphi_1),\varphi_2,\varphi_3,r(\varphi_3)]$. This
  is indeed a simplex in $P$ since 
  $\varphi_1 \to \varphi_2$ implies $\varphi_1 \to r(\varphi_2)$
  and $r(\varphi_1) \to \varphi_2$ by the properties of $\cong$.
\end{proof}

Notice that on $U$ the relation $\rightarrow$ defines a partial order
on the vertices which is a total order when restricted to simplices.

$U$ has the structure of a simplicial set with degeneration map given
by
\begin{align*}
  [\varphi_1,\dots,\varphi_i,\dots,\varphi_n] \mapsto
  [\varphi_1,\dots,\varphi_i,\varphi_i,\dots,\varphi_n].
\end{align*}
The geometric realization of this simplicial set, where degenerate
simplices are identified with the corresponding lower dimensional
simplices, is homeomorphic to the geometric realization of the
ordered simplicial complex $W$ whose simplices are of the form
$[\varphi_1,\dots,\varphi_n]$ where $\varphi_1,\dots,\varphi_n$ are
pairwise different and $\varphi_i \to \varphi_j$ for $i < j$.

Now by \cite[Proposition 2.1]{rourkeDSETSHOMOTOPYTHEORY1971} the
geometric realization of $U$ as a simplical set and the geometric
realization of $U$ as a $\Delta$-set are homotopy equivalent. 
Since $P$, $U$ and $W$ are homotopy equivalent, it
is enough to show that $W$ is contractible.

We already know that $W$ is simply connected by \Cref{thm:simply-connected}.  By
Whiteheads theorem it is therefore enough to show that all higher
homology groups vanish.  For the computation we use simplicial
homology for simplicial complexes
as defined in \cite[Chapter 7]{rotmanIntroductionAlgebraicTopology1998}.

Given a homology class of $W$ we will find a series
of representatives that contain only simplices
whose vertices can be represented as the refinement
of more and more different vertices from a finite set.
Eventually all of these simplices will therefore be
degenerate and the homology class must be zero.

\begin{defn}
  \label{def:Ktag}
	Let $K$ be a finite subcomplex of $W$. Denote by $K'$ the subcomplex
	of $W$ whose $n$-simplices are either of the form
   \begin{enumerate}[(I)]
     \item 
	$[\tilde{\delta}
     (\varphi_{i_0},\varphi_{j_0}), \dots,
     \tilde{\delta}(\varphi_{i_n},\varphi_{j_n})]$
     with $i_k < j_k$ for all $k \in \{0,\dots,n\}$,
     $i_0  \leq \dots\leq i_n$ and $j_0\leq \dots\leq j_n$, or
     \label{enum:typei}
 \item $[
   \tilde{\delta}(\varphi_{0},\varphi_{\ell}),
   \tilde{\delta}(\varphi_{1},\varphi_{\ell}),
   \dots,
   \tilde{\delta}(\varphi_{\ell},\varphi_{\ell}),
   \tilde{\delta}(\varphi_{\ell},\varphi_{\ell+1}),
   \dots,
   \tilde{\delta}(\varphi_{\ell},\varphi_{n})]$ for some $\ell \in
   \{1,\dots,n\}$
   \label{enum:typeii}
 \end{enumerate}
   where $[\varphi_0,\dots,\varphi_n]$ is a simplex in $K$.
\end{defn}

The following theorem is based on the Freudenthal subdivision of
a simplex, see \Cref{sec:freudenthal} for a proof and more background information.
\begin{thm}\label{thm:freudenthal-subdivision-operator}
  Let $K$ be a finite subcomplex of $W$. There is map
  between chain complexes $C_n(W) \to C_n(W)$
  mapping $C_n(K)$ into $C_n(K')$ which is chain homotopic to the identity.
\end{thm}

\begin{thm}\label{thm:trivial-homology}
  For every $n\geq 1$ the simplicial homology group $H_n(W)$ is
  trivial.
\end{thm}
\begin{proof}
  Let $[\alpha]$ be a homology class in $H_n(W)$.  There is a finite 
  subcomplex $K$ of $W$ such that all simplices of $\alpha$ are
  contained in $K$. Let $\tilde{K}$ be
  the maximal subcomplex of $W$ containing all vertices of the form
  $\tilde{\delta}(\varphi_1,\dots,\varphi_\ell)$ for
  $\varphi_1,\dots,\varphi_\ell \in K$, $\ell \in \N$.
  
  For a vertex $\varphi$ in $\tilde{K}$ let $k$ be the
  maximal number such that
  $\varphi \cong \tilde{\delta}(\varphi_1,\dots,\varphi_k)$ for pairwise
  different vertices $\varphi_1,\dots,\varphi_k$ in $K$.  We call $k$
  the \emph{rank} of $\varphi$ with respect to $K$.
	
  Let $K_n$ be the maximal subcomplex of $\tilde{K}$ such that the rank of every
  vertex is at least $n$.
	
  The result now follows directly from the following two claims.
	\begin{enumerate}[(1)]
		\item 	$K_{|K|+1}$ is empty.
        \label{enum:claim-homology-1}
      \item For every cycle $\alpha$ in $K_{k}$ with $1\leq k$ there
        is a cycle $\alpha'$ in $K_{k+1}$ representing the same
        homology class as $\alpha$ in $H_n(W)$.
        \label{enum:claim-homology-2}
	\end{enumerate}
   The first claim follows directly from the definition of rank.
	For the second claim assume $\alpha$ is a cycle
   supported in $K_{k}$. By
   \Cref{thm:freudenthal-subdivision-operator}
   we can find $\beta \in K_k'$ in the same homology class as $\alpha$.
	Recall that there are two types of simplices in $K'_k$.
   Simplices of type \ref{enum:typei} are good as
   they are already contained in $K_{k+1}$. We just have to get
   rid of the simplices of type \ref{enum:typeii}.
	Let $V$ be the set of vertices of simplices appearing in $\beta$
   whose rank is at most $k$. They can only appear
   in simplices of type \ref{enum:typeii} and each of
   these simplices contains at most one such vertex. Let $\varphi \in V$.
   For a chain $\gamma$ denote by $\gamma^+$
   the part of the chain consisting only of simplices containing
   $\varphi$ and let $\gamma^-$ be the part of the chain consisting
   only of simplices not containing $\varphi$. We then get a
   decomposition $\gamma=\gamma^+ + \gamma^-$.

   Let $I$ be the set of all
   vertices $\psi$ occurring in simplices of $\beta^+$ with
   $\psi \to \varphi$ and $\psi \neq \varphi$.  Similarly let $T$
   be the set of all vertices $\psi$ occurring in simplices of
   $\beta^+$ with $\varphi \to \psi$ and
   $\psi \neq \varphi$.
	
	We have to differentiate between two cases.
	
	Case 1: $I \neq \emptyset$.  Let
   $\tilde{\varphi}=\tilde{\delta}(\psi_1,\dots,\psi_m,\varphi)$ for
   $I = \{\psi_1,\dots,\psi_m\}$. This is
   defined by  by \Cref{lem:arrow-right-grouping}.
   Since every simplex in $\beta$ contains at
   most one vertex of rank $n$ and all other vertices have larger
   rank, the rank of $\tilde{\varphi}$ is larger then $n$. Define maps
   $D_{\tilde{\varphi}}: C_n(\supp({\beta}^+)) \to C_{n+1}(W)$ and
   $D_{\tilde{\varphi}}: C_{n-1}(\supp(\partial(\beta^+))) \to
   C_{n}(W)$, by
   $D_{\tilde{\varphi}}([\varphi_1,\dots,\varphi_k]) =
   [\tilde{\varphi},\varphi_1,\dots,\varphi_k]$.  
   Next we show that there actually is such a simplex
   $[\tilde{\varphi},\varphi_1,\dots,\varphi_k]$ in $W$.
   Let
   \[[\tilde{\delta}(\varphi_0,\varphi_\ell),\dots,
     \tilde{\delta}(\varphi_{\ell},\varphi_{\ell}),
   \tilde{\delta}(\varphi_{\ell},\varphi_{\ell+1}),
   \dots,
   \tilde{\delta}(\varphi_{\ell},\varphi_{n})]\]
 
   be a simplex of type \ref{enum:typeii} appearing in $\beta$ with
   $\varphi_\ell=\varphi$.
   For $0\leq k<\ell$ we have $\tilde{\delta}(\varphi_k,\varphi_\ell) \in I$
   and therefore $\tilde{\varphi} \to
   \tilde{\delta}(\varphi_k,\varphi_\ell)$.
   For $\ell<k\leq n$ we have $\varphi=\varphi_\ell \to
   \tilde{\delta}(\varphi_\ell,\varphi_k)$
   and thus again $\tilde{\varphi} \to
   \tilde{\delta}(\varphi_\ell,\varphi_k)$.

   A simple calculation  gives
	\begin{align*}
	\partial D_{\tilde{\varphi}} (\beta^+)
	= \beta^+ - D_{\tilde{\varphi}} (\partial (\beta^+)).
	\end{align*}
	Now $(\partial (\beta^+))^+$ must be zero as
   $\partial(\beta)=0$ and
   $(\partial(\beta))^+ = (\partial(\beta^+))^+$.
   Hence $(D_{\tilde{\varphi}}(\partial(\beta^+)))^+=(D_{\tilde{\varphi}}(\partial(\beta^+)^+))=0$
   and $D_{\tilde{\varphi}}(\partial(\beta^+))=(D_{\tilde{\varphi}}(\partial(\beta^+)))^-$.
   Therefore
   $\beta$ is in the same homology class as
   $\beta^- + (D_{\tilde{\varphi}} (\partial({\beta}^+)))^-$.  This
   cycle does not contain any simplex containing $\varphi$ and every
   simplex in $(D_{\tilde{\varphi}}(\partial(\beta^+)))^-$ has rank at
   least $k+1$.  Replace $\beta$ by
   $\beta^-+(D_{\tilde{\varphi}} (\partial(\beta^+)))^-$.  Now
   repeat this process for all vertices $\varphi \in V$.  The
   resulting cycle contains only vertices of rank $k+1$.
	
	Case 2: $I=0$. In the case $T \neq 0$ and instead of adding a
   vertex to the beginning of simplices we add it to the end. More
   precisely, repeat the argument of Case 1 using
   $\tilde{\varphi} := \delta(\psi_1,\dots,\psi_m,\varphi)$ for
   $\{\psi_1,\dots,\psi_m\}=T$ and
   $D_{\tilde{\varphi}}([\varphi_1,\dots,\varphi_n]):=(-1)^n
   [\varphi_1,\dots,\varphi_n,\tilde{\varphi}]$.
\end{proof}
\appendix

\section{The Freudenthal Subdivision}
\label{sec:freudenthal}
The Freudenthal sudivision of a simplex is based on the following
idea. First we subdivided the cube $[0,1]^n$ into the simplices of the form
\begin{align*}
  \{x \in [0,1]^n \setsep 0 \leq x_{\pi(1)} \leq \dots \leq x_{\pi(n)} \leq 1\} 
\end{align*}
where $\pi$ varies over all permutations of $\{1,\dots,n\}$.
We lift this to a subdivision of $\R^n$ be subdividing all
lattice cubes $z+[0,1]^n, z \in \Z^n$ in the same way.
Finally this subdivision induces a subdivision of the simplex
\begin{align*}
  \simplex_n := \{x \in [0,2]^n \setsep 0 \leq x_n \leq \dots \leq x_1 \leq 2\}
\end{align*}
and this is the subdivision we are looking for.
It was introduced by Freudenthal in
\cite{freudenthalSimplizialzerlegungenBeschrankterFlachheit1942},
and reappeared in several contexts, see
for example \cite{edelsbrunnerEdgewiseSubdivisionSimplex2000},
\cite{bokstedtCyclotomicTraceAlgebraic1993}
and \cite{brunSubdivisionsToricComplexes2005}.

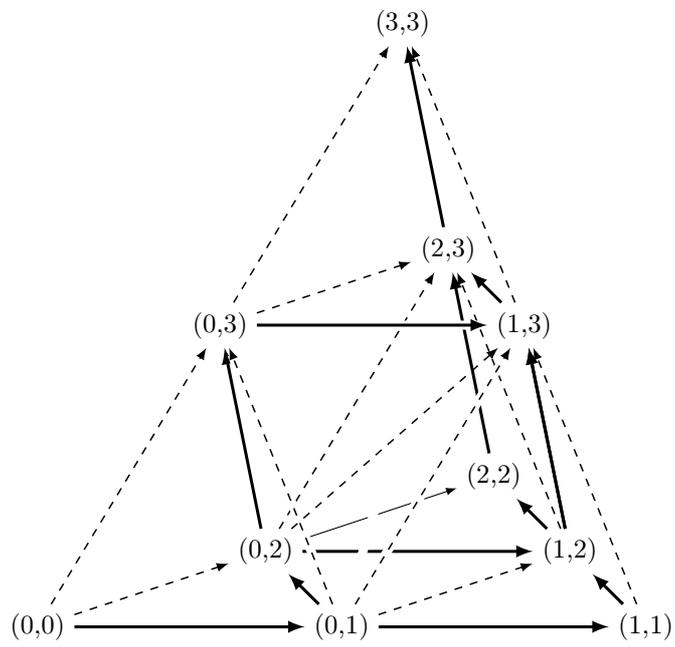
\begin{figure}
\begin{center}
\begin{tikzpicture}[scale=4.0,
  mainline/.style = {line width = 0.4mm,-latex},
  whiteline/.style = {line width = 1.2mm, white, shorten <=0.1cm,shorten >=0.1cm},
  thinline/.style = {line width = 0.2mm, dashed,-latex},
  ]
  \node (00) at (0,0) {(0,0)};
  \node (11) at (2,0) {(1,1)};
  \node (22) at (1.5,0.5) {(2,2)};
  \node (33) at (1.2,2) {(3,3)};
  \node (01) at ($(00)!0.5!(11)$) {(0,1)};
  \node (02) at ($(00)!0.5!(22)$) {(0,2)};
  \node (03) at ($(00)!0.5!(33)$) {(0,3)};
  \node (12) at ($(11)!0.5!(22)$) {(1,2)};
  \node (13) at ($(11)!0.5!(33)$) {(1,3)};
  \node (23) at ($(22)!0.5!(33)$) {(2,3)};
  \draw[mainline] (00) -- (01);
  \draw[mainline] (01) -- (11);
  \draw[thinline] (00) -- (02);
  \draw[thin, -latex] (02) -- (22);
  \draw[mainline] (11) -- (12);
  \draw[mainline] (12) -- (22);
  \draw[mainline] (22) -- (23);
  \draw[mainline] (23) -- (33);
  \draw[thinline] (11) -- (13);
  \draw[thinline] (13) -- (33);
  \draw[thinline] (00) -- (03);
  \draw[thinline] (03) -- (33);
  \draw[mainline] (02) -- (03);
  \draw[mainline] (01) -- (02);
  \draw[mainline] (13) -- (23);
  \draw[mainline] (12) -- (13);
  \draw[mainline] (12) -- (13);
  \draw[mainline] (02) -- (12);
  \draw[thinline] (02) -- (23);
  \draw[thinline] (12) -- (23);
  \draw[whiteline] (01) -- (03);
  \draw[thinline] (01) -- (03);
  \draw[thinline] (03) -- (23);
  \draw[thinline] (01) -- (12);
  \draw[whiteline] (01) -- (13);
  \draw[thinline] (01) -- (13);
  \draw[whiteline] (01) -- (23);
  \draw[thinline] (02) -- (13);
  \draw[whiteline] (03) -- (13);
  \draw[mainline] (03) -- (13);
\end{tikzpicture}
\end{center}
\caption{Freudenthal subdivision of the $3$-dimensional simplex, every
  directed path of length three of thick edges, of which no two are in
  a line, determines one cell of the subdivision}
\label{fig:simplex-freudenthal}
\end{figure}

Formally, let $V:=\{(i,j) \in \{0,\dots,n\}^2 \setsep i \leq j\}$ and
define a map
\begin{align*}
  \vartheta: V \to \Z^n,\quad \vartheta(i,j)_k =
  \begin{cases}
    2 & \text{for } k \leq i \\
    1 & \text{for } i < k \leq j \\
0 & \text{for } j < k 
  \end{cases}.
\end{align*}
So for example
\begin{align*}
  \vartheta(0,0)&=(0,\dots,0),\\
  \vartheta(1,3)&=(2,1,1,0,\dots,0) \text{ and }\\
  \vartheta(1,n)&=(2,1,\dots,1).
\end{align*}
This map is clearly injective.
Notice that the points $\theta(k,k)$ for $k \in \{1,\dots,n\}$
are precisely the vertices of $\simplex_n$
and that $\vartheta(i,j)$ is the midpoint of $\vartheta(i,i)$ and $\vartheta(j,j)$.

Now let $\simplex_n^F$ be the simplicial complex generated by
all simplices $[v_0,\dots,v_n]$ with $v_i \in \{0,1,2\}^n \cap \simplex_n$
and
\begin{align}\label{eq:permutation-freudenthal}
      v_{i}=v_{i-1} + e_{\pi(i)}
\end{align}
for all $i \in \{1,\dots,n\}$ where
$\pi$ is a permutation of $\{1,\dots,n\}$.
We also denote the set of $n$-dimensional simplices in this subdivision
by $\simplex_n^F$.

Define
\begin{align*}
  \sgn [v_0,\dots,v_n]=\sgn(\pi) &= \det
  (v_1-v_0,v_2-v_1,\dots,v_n-v_{n-1})\\
  &=\det (v_1-v_0,v_2-v_0,\dots,v_n-v_0).
\end{align*}
Such a simplex is uniquely determined by $v_0$ and $\pi$.
The only restriction on $v_0$ is $v_0 \in \{0,1\}^n$
but depending on $v_0$ not every permutation $\pi$ produces vertices
$v_i \in \simplex_n$. The complex
$\simplex^F_n$ is the so called \emph{Freudenthal subdivision} of the simplex $\simplex_n$.
See \Cref{fig:simplex-freudenthal} for an illustration
of $\simplex^F_3$.


The inclusion maps of the $k$-th face of $\simplex_n$ are explicitly given by
doubling the $k$-th coordinate, i.e., 
\begin{align*}
  d_k&: \simplex_{n} \to \simplex_{n+1},\\
  d_k(x_1,\dots,x_n)&=(x_1,\dots,x_k,x_k,\dots,x_n) \text{ for } 1\leq
  k \leq n, \\
  d_0(x_1,\dots,x_n)&=(2,x_1,\dots,x_n),\\
  d_{n+1}(x_1,\dots,x_n)&=(x_1,\dots,x_n,0).
\end{align*}

Based on this subdivision of $\simplex_n$ we are now subdividing an
(abstract) ordered simplicial complex $K$.
Let $K^F$ be the simplicial complex given by
simplices of the form $[(v_{i_0},v_{j_0}), \dots,(v_{i_m},v_{j_m})]$
where $[\vartheta(i_0,j_0),\dots,\vartheta(i_m,j_m)]$ is
a simplex in $\simplex_n^F$ and $[v_0,\dots,v_n]$ is a simplex in $K$.
In particular we must have $i_k \leq j_k$ for all $k \in
\{1,\dots,m\}$,
$i_0\leq\dots\leq i_m$ and $j_0\leq\dots\leq j_m$.
For a simplex $T=[(\vartheta(i_0,j_0),\dots,\vartheta(i_n,j_n)] \in
\simplex_n^F$ and a simplex $[v_0,\dots,v_n] \in K$
define
$\tau_{v_0,\dots,v_n}(T):=[(v_{i_0},v_{j_0}),\dots,(v_{i_n,j_n})] \in K^F$.

To construct a chain homotopy between the identity and the
subdivision, which we need for the proof of \Cref{thm:freudenthal-subdivision-operator},
we also construct the simplicial complex $K^\#$ whose simplices of
maximal dimension are of
the form
\[[v_{i_0},\dots,v_{i_k},(v_{i_{k}},v_{j_{k}}),\dots,(v_{i_m},v_{j_m})]\] for $k \in
\{0,\dots,m\}$ where $[(v_{i_0},v_{j_0}),\dots,(v_{i_m},v_{j_m})] \in
K^F$ and $v_0,\dots,v_k$ are pairwise different.
Both $K$ and $K^F$ are subcomplexes of $K^\#$.

Now define a family of maps $F: C_m(K) \to C_m(K^\#)$ by
\begin{align}
  [v_0,\dots,v_m] &\mapsto
                    \sum_{T \in \simplex^F}
                    \sgn(T) \tau_{v_0,\dots,v_m}(T).
  \label{eq:F}
\end{align}
Next we show that $F$ is a chain map, i.e., that it intertwines
with the boundary map.
We have
\begin{align*}
  \partial F([v_0,\dots,v_m]) &= \sum_{T \in \simplex_m^F} 
                                \sgn(T) \partial \tau_{v_0,\dots,v_m}(T), \\
  F\partial ([v_0,\dots,v_m]) &=\sum_{k=0}^m (-1)^k\sum_{S \in \simplex_{m-1}^F}
                                \sgn(S)
                                \tau_{v_0,\dots,\hat{v_k},\dots,v_m}(S)
  \\
                              &=\sum_{k=0}^m (-1)^k\sum_{S \in \simplex_{m-1}^F}
                                \sgn(S)
                                \tau_{v_0,\dots,v_m}(d_k(S)).
\end{align*}                       
It is therefore enough to show
\begin{align}
  \sum_{T \in \simplex_m^F} \sgn(T) \partial T=\sum_{k=0}^m (-1)^k \sum_{S \in \simplex_{m-1}^F}
                                \sgn(S)
                                d_k(S). \label{eq:chain-map}
\end{align}

We now expand the left hand side as
\begin{align*}
  \sum_{T \in \simplex_m^F} \sgn(T) \partial T=\sum_{[w_0,\dots,w_m] \in \simplex_m^F}
  \sgn([w_0,\dots,w_m])
  \sum_{\ell}
  (-1)^\ell
  [w_0,\ldots,
  \widehat{w_\ell},\dots,
  w_m].
\end{align*}
A simplex $[w_0,\ldots,
  \widehat{w_\ell},\dots,
  w_m]$
  appears exactly once in this sum if
  \begin{enumerate}[(a)]
   \item $0<\ell<m$ and $w_{\ell+1}-w_{\ell-1} = e_k+e_{k+1}$ for some
     $k \in \{1,\dots,m-1\}$, or
   \item $\ell=0$ and one of the entries of $w_1$ equals $2$, or 
   \item $\ell=m$ and one of the entries of $w_{m-1}$ equals $0$.
   \end{enumerate}
   Otherwise $[w_0,\ldots, \widehat{w_\ell},\dots, w_m]$ appears twice
   with opposite signs and hence these terms cancel out. For example
   if $w_1$ contains only $0$ and $1$ as entries, then
\begin{align*}
[\widehat{w_0},w_1,\ldots, w_m] = [w_1,\ldots,w_m,
  \widehat{w_{m+1}}] \text{ for }w_{m+1}:=w_1+e_1+\dots+e_m
\end{align*}
and $\sgn([w_0,\dots,w_m])=(-1)^m\sgn([w_1,\dots,w_{m+1}])$,
hence \[(-1)^0 \sgn([w_0,\dots,w_m]) = - (-1)^{m+1} \sgn([w_1,\dots,w_{m+1}]).\]
For a simplex $[w_0,\dots,\widehat{w_\ell},\dots,w_m]$ with
$0<\ell<m$ and $w_{\ell+1}-w_{\ell-1}=e_k+e_{k+1}$
the vectors $w_0,\dots,w_{\ell-1},w_{\ell+1},\dots,w_m$
all have the property that there $k$-th and $k+1$-th entry agree, and hence
$[w_0,\dots,\widehat{w_\ell},\dots,w_m]$
appears as a summand of the form $d_k(T)$ on the right hand side of \eqref{eq:chain-map}.
In a simplex $[\widehat{w_0},w_1\dots,w_m]$ with a $0$-entry
in $w_1$, all the vectors $w_1, \dots, w_n$ must be
$0$ in the last coordinate, hence they appear as a
summand $d_m(T)$ on the right hand side of \eqref{eq:chain-map}.
Finally every simplex $[w_0,\dots,w_{m-1},\widehat{w_m}]$
appears as a summand of the form $d_0(T)$ on the right hand side of
\eqref{eq:chain-map}. Checking the signs, this shows that both sides of \eqref{eq:chain-map} are equal.

Define a map $\varrho: C(K^F) \to C(K^\#)$ via
\begin{align}
  [(v_{i_0},v_{i_1}),\dots,(v_{i_m,j_m})] \mapsto
  \sum_{k_0}^m (-1)^k [v_{i_0},\dots,v_{i_k},(v_{i_k},v_{j_k}),\dots,(v_{i_m},v_{j_m})]
  \label{eq:rho}
\end{align}
Remember that all simplices with repeated vertices are trivial in $C(K^\#)$.

The chain map $F$ is chain homotopic to the identity
$\id: C(K) \to C(K^\#)$ via the map given by $\varrho \circ F$.
To see this we first compute $\partial \circ \varrho + \varrho \circ
\partial$. Under this map
\begin{align*}
   [(v_{i_0},v_{j_0}), \dots,(v_{i_m},v_{j_m})]
   \mapsto
   &\sum_{k} \sum_{\ell \leq k}
   (-1)^{k+\ell}
   [v_{i_0},\dots,\hat{v}_{i_\ell},\dots,,v_{i_k},(v_{i_k},v_{j_k}),\dots,(v_{i_m},v_{j_m})] \\
   +&\sum_{k}
   \sum_{\ell \geq k} (-1)^{k+\ell+1}
   [v_{i_0},\dots,v_{i_k},(v_{i_k},v_{j_k}),\dots,
   \widehat{(v_{i_\ell},v_{j_\ell})},\dots,(v_{i_m},v_{j_m})] \\
   + &\sum_{\ell} \sum_{\ell <  k}
   (-1)^{k+1+\ell}
   [v_{i_0},\dots,\hat{v}_{i_\ell},\dots,,v_{i_k},(v_{i_k},v_{j_k}),\dots,(v_{i_m},v_{j_m})] \\
   +&\sum_{\ell}
   \sum_{k > \ell} (-1)^{k+\ell}
  [v_{i_0},\dots,v_{i_k},(v_{i_k},v_{j_k}),\dots,
  \widehat{(v_{i_\ell},v_{j_\ell})},\dots,(v_{i_m},v_{j_m})]  \\
      =&\sum_{k} 
   [v_{i_0},\dots,v_{i_{k-1}},(v_{i_k},v_{j_k}),\dots,(v_{i_m},v_{j_m})] \\
   -&\sum_{k}
   [v_{i_0},\dots,v_{i_k},(v_{i_{k+1}},v_{j_{k+1}}),\dots,(v_{i_m},v_{j_m})]
  \\
  =&[(v_{i_0},v_{j_0}),\dots,(v_{i_m,j_m})]-[v_{i_0},\dots,v_{i_m}] 
\end{align*}

Now
$[\vartheta(i_0,j_0),\dots,\vartheta(i_m,j_m)]=[\vartheta(0,0),\dots,\vartheta(m,m)]$
is the unique
simplex for which $i_0,\dots,i_m$ are pairwise different. The sign of
this simplex is $1$. Hence
\begin{align*}
  (\partial\circ \varrho \circ F + \varrho \circ F\circ \partial)
  ([v_0,\dots,v_m])
  &=(\partial\circ \varrho  + \varrho \circ \partial)(F([v_0,\dots,v_m])) \\
  &=F([v_0,\dots,v_m])-[v_0,\dots,v_m].
\end{align*}

This shows that $F$ and $\id$ are chain homotopic. We now return to
the simplicial complex $W$ from \Cref{sec:contractability}.
Let $K$ be a finite subcomplex of $W$.
 The map $\tilde{\delta}$ induces a simplicial map from $W^\# \to W$
given by $\varphi \mapsto \varphi$ and $(\varphi_1,\varphi_2) \mapsto \tilde{\delta}(\varphi_1,\varphi_2)$.
This expression is defined since for a vertex $(\varphi_1,\varphi_2) \in W^\#$
we have $\varphi_1 \to \varphi_2$, hence $(\varphi_1,\varphi_2) \in H_2$.
It is a simplicial map since
every simplex in $W^\#$ is of the form
$[\varphi_{i_0},\dots,\varphi_{i_k},(\varphi_{i_{k}},\varphi_{j_{k}}),
\dots,(\varphi_{i_n},\varphi_{j_n})]$ where
$[\varphi_1,\dots,\varphi_p]$ is a simplex in $W$,
$i_0\leq i_1\leq \dots \leq i_n$ and $j_{k+1} \leq j_{k+2}\leq
\dots\leq j_m$ and $i_\ell \leq j_\ell$, 
hence
\begin{align*}
  \varphi_{i_\ell} &\to \varphi_{i_m} && \text{for } 0\leq \ell<m\leq k,\\
  \varphi_{i_\ell} &\to
                     \tilde{\delta}(\varphi_{i_{m}},\varphi_{j_{m}})
                                      && \text{for } 0\leq \ell\leq k <  m \leq n,\\
  \tilde{\delta}(\varphi_{i_\ell},\varphi_{j_\ell})
                   &\to
                     \tilde{\delta}(\varphi_{i_{m}},\varphi_{j_{m}})
                     && \text{for } k<\ell<m\leq n.
\end{align*}  
This furthermore induces a map of chain complexes $D: C_n(W^\#) \to C_n(W)$.
Restricted to $W$ it is the identity.

Clearly every simplex in $\simplex_n^F$ contains at most one vertex of
$\simplex_n$ and every vertex of $\simplex_n$ is contained in
precisely one simplex of $\simplex_n^F$.  These simplices are of the
form
$[\vartheta(0,j),\dots,
\vartheta(j-1,j),
\vartheta(j,j),
\vartheta(j,j+1),\dots,
\vartheta(j,n)]$.
The image of $C_n(W^F)$ under $D$ is therefore contained in $C_n(W')$ (see \Cref{def:Ktag}).
Composing $F: C_n(W) \to C_n(W^\#)$ with $D$ thus gives a map
from $C_n(W) \to C_n(W)$ which maps $C_n(K)$ to $C_n(K')$ and which is chain-homotopic to
the identity. We thus proved \Cref{thm:freudenthal-subdivision-operator}.

\bibliographystyle{alpha}
\bibliography{references}
\end{document}